\documentclass{siamart251104}
\usepackage{lipsum}
\usepackage{amsfonts}
\usepackage{amsmath}
\usepackage{bm}
\usepackage{graphicx}
\usepackage{epstopdf}
\usepackage{etoolbox}
\usepackage{nicefrac}
\usepackage{mathtools}
\usepackage{pgfplots}
\usepackage{tikz}
\usepackage{caption}
\usepackage{multirow}
\usepackage{ragged2e}
\usepackage{float}
\usepackage{flafter}
\usepackage{subcaption}
\usepackage[normalem]{ulem}
\usepackage{comment}
\usetikzlibrary{patterns}
\usetikzlibrary{patterns.meta}
\usetikzlibrary{spy}
\usetikzlibrary{shapes.misc}
\usepgfplotslibrary{fillbetween}
\usepackage{algpseudocode}
\usetikzlibrary {arrows.meta}

\definecolor{myred}{RGB}{255, 71, 26}
\definecolor{myblue}{RGB}{51, 153, 255}
\definecolor{mygreen}{RGB}{0, 204, 102}
\definecolor{myorange}{RGB}{255, 200, 0}
\definecolor{mypurple}{RGB}{186, 85, 211}

\newsiamremark{remark}{Remark}

\usepackage{ifoddpage} 

\newtoggle{not_resized}
\togglefalse{not_resized}

\ifpdf%
  \DeclareGraphicsExtensions{.eps,.pdf,.png,.jpg}
\else
  \DeclareGraphicsExtensions{.eps}
\fi
\usepackage{amsopn}
\DeclareMathOperator{\diag}{diag}
\usepackage{booktabs}

\newcommand{\TheTitle}{Chebyshev Accelerated Subspace Eigensolver for Pseudo-hermitian Hamiltonians}

\newcommand{\TheShortTitle}{%
ChASE for Pseudo-hermitian Hamiltonians
}

\newcommand{\TheShortNames}{%
  E. DI NAPOLI, C. RICHEFORT, X. WU
}

\author{Edoardo Di Napoli\thanks{Jülich Supercomputing Centre, Forschungzentrum Jülich, Germany (\email{e.di.napoli@fz-juelich.de}, \email{c.richefort@fz-juelich.de}, \email{xin.wu@fz-juelich.de})} \and Clément RICHEFORT\footnotemark[1] \and  Xinzhe Wu\footnotemark[1]}

\title{{\TheTitle}}
\headers{\TheShortTitle}{\TheShortNames}
\ifpdf%
\hypersetup{%
  pdftitle={\TheTitle},
  pdfauthor={Clément RICHEFORT}
}
\fi

\def\Lcal{{\mathcal{L}}}
\def\Ocal{{\mathcal{O}}}
\def\Fcal{{\mathcal{F}}}

\def\be{{\bm{e}}}
\def\br{{\bm{r}}}
\def\bx{{\bm{x}}}
\def\by{{\bm{y}}}
\def\bz{{\bm{z}}}

\def\bv{{\bm{v}}}
\def\bu{{\bm{u}}}

\def\bw{{\bm{w}}}

\def\by{{\bm{y}}}
\def\nev{{\texttt{nev}}}
\def\nex{{\texttt{nex}}}
\def\nevex{{\texttt{nevex}}}
\def\deg{{\texttt{deg}}}

\def\tol{{\texttt{tol}}}
\def\locked{{\texttt{locked}}}

\def\cond{{\textup{cond}}}

\newcommand\ee{\end{equation}}
\newcommand\bea{\begin{eqnarray}}
\newcommand\eea{\end{eqnarray}}
\newcommand\bi{\begin{itemize}}
\newcommand\ei{\end{itemize}}

\newcommand\eqalign[1]{%
	\vcenter{%
		\normalbaselines \advance\baselineskip 5pt
		\advance\lineskip 5pt \tabskip=0pt
		\halign{%
			&\hfil $\displaystyle{##{}}$&
			$\displaystyle{{}##}$\hfil\cr
			#1\crcr
			}%
		}%
	}

\newcommand\dotline{\par\hbox to \hsize{\dotfill}\par}
\makeatletter
\def\befored@t#1.#2.#3;{#1}
\def\afterd@t#1.#2.#3;{#2}
\def\refhead#1{\edef\next{\ref{#1}}\expandafter\befored@t\next..;}
\def\reftail#1{\edef\next{\ref{#1}}\expandafter\afterd@t\next..;}
\def\lsim{\mathrel{\mathpalette\@versim<}}
\def\gsim{\mathrel{\mathpalette\@versim>}}
\def\@versim#1#2{\vcenter{\offinterlineskip
        \ialign{$\m@th#1\hfil##\hfil$\crcr#2\crcr\sim\crcr } }}
\newcommand\becomes[1]{\mathchoice{\becomes@\scriptstyle{#1}}
   {\becomes@\scriptstyle{#1}} {\becomes@\scriptscriptstyle{#1}}
   {\becomes@\scriptscriptstyle{#1}}}
\def\becomes@#1#2{\mathrel{\setbox0=\hbox{$\m@th #1{\,#2\,}$}%
        \mathop{\hbox to \wd0 {\rightarrowfill}}\limits_{#2}}}
\makeatother

\newcommand{\defeq}{\vcentcolon=}
\newcommand{\rw}[1]{#1}
\newcommand{\wurw}[1]{#1}

\newcommand{\XZdone}[1]{}

\newcommand{\EDdone}[1]{}

\newcommand{\CRdone}[1]{}

\newtoggle{versiontwo}
\toggletrue{versiontwo}

\newcommand{\NEW}[1]{%
  \iftoggle{versiontwo}{\textcolor{black}{#1}}{}%
}

\newcommand{\OLD}[1]{%
  \iftoggle{versiontwo}{}{#1}%
}

\begin{document}

\maketitle


\begin{abstract}
Studying the optoelectronic structure of materials can require the computation of several thousands of the smallest positive eigenpairs of a pseudo-hermitian Hamiltonian. Iterative eigensolvers may be preferred over direct methods for this task since their complexity is a function of the desired fraction of the spectrum. In addition, they generally rely on highly optimized and scalable kernels such as matrix-vector multiplications that leverage the massive parallelism and the computational power of modern exascale systems. The \textit{Chebyshev Accelerated Subspace iteration Eigensolver} (ChASE) is able to compute several thousands of the most extreme eigenpairs of dense hermitian matrices with proven scalability over massive parallel accelerated clusters \cite{wu_advancing_2023}. This work presents an extension of ChASE to solve for a portion of the \OLD{spectrum}\NEW{smallest positive eigenpairs} of pseudo-hermitian Hamiltonians as they appear in the treatment of excitonic materials. By exploiting the numerical structure and spectral properties of the Hamiltonian matrix, we \NEW{preserve the characteristic positive-negative symmetry in the treatment of the eigenvectors and }propose an oblique variant of Rayleigh-Ritz projection \OLD{featuring}\NEW{that features} quadratic convergence of the Ritz values with no explicit construction of the dual basis. Additionally, we introduce a parallel implementation of the recursive matrix-product operation appearing in the Chebyshev filter with limited amount of global communications. Our development is supported by a full numerical analysis and experimental tests.
\end{abstract}

\begin{keywords}
 Pseudo-hermitian solver, Excitonic Hamiltonian, Bethe-Salpeter equation, Chebyshev filter, Oblique Rayleigh-Ritz, Quadratic convergence, GPU-enabled eigensolver
\end{keywords}
\section{Introduction}
The optical properties of materials under light stimulation are typically studied through the \textit{Bethe-Salpeter Equation} (BSE) \cite{bechstedt_many-body_2015}. For numerical solution the BSE is expressed as an eigenvalue problem represented by an excitonic Hamiltonian.
Such an Hamiltonian is of the form
\begin{equation}
    H \defeq \begin{bmatrix}
        A & B\\
        -\bar{B} & -\bar{A}
    \end{bmatrix} \qquad \text{with} \qquad A = A^* \quad \text{and} \quad B = B^T.
    \label{introduction::hamiltonian_definition}
\end{equation}
The $m \times m$ dense blocks $A$ and $B$ are respectively referred to as resonant and coupling terms. The two blocks $\bar{A}$ and $\bar{B}$ simply stand for the conjugate of $A$ and $B$. The Hamiltonian $H$ is dense, has size $n \defeq 2m$, and is termed a pseudo-hermitian matrix, as it satisfies the relation
\begin{equation}
    SH = H^*S \qquad \text{with} \qquad S \defeq \begin{bmatrix}
        I & 0 \\
        0 & -I
    \end{bmatrix}.
    \label{introduction::pseudo_hermitian_condition}
\end{equation}
This paper focuses on the most common case in which the eigenproblem 
\begin{equation}
    H\bv = \lambda\bv \quad \Leftrightarrow \quad SH\bv = \lambda S\bv \quad \text{with} \quad \bv \in \mathbb{C}^n \quad \text{and} \quad \lambda \in \mathbb{R}
    \label{eqn::form_of_eigenproblem}
\end{equation}
is ``definite'', meaning that $SH$ is \textit{hermitian positive definite} (HPD), i.e.,
\begin{equation}
    SH = \begin{bmatrix}
        A & B\\
        \bar{B} & \bar{A}
    \end{bmatrix}  \succ 0.
    \label{introduction::quasi_hermitian_condition}
\end{equation}
We will further see that the eigenvalues $\lambda$ in  \eqref{eqn::form_of_eigenproblem} are real-valued because of the positive-definiteness of $SH$. When the coupling-term $B$ can be neglected, then the pseudo-hermitian eigenproblem reduces to an hermitian eigenvalue problem for which the list of modern numerical libraries is extensive \cite{blackford_scalapack_1997,hernandez_slepc_2005,imamura_2011_eigenexa_development,marek_elpa_2014,winkelmann_chase_2019,gates_evolution_2025}. However, this approximation is not always possible and can lead to inaccurate simulation of the desired optical properties of the materials. Therefore, developing fast and scalable eigensolvers for the original pseudo-hermitian eigenproblem can have a potential important impact in the community of materials scientists studying optoelectronics. While direct eigensolvers such as ELPA introduced variants for this case \cite{penke_high_2020,penke_efficient_2022}, modern iterative eigensolvers cannot yet extract efficiently more than a few hundreds eigenpairs \cite{alvarruiz_variants_2025,milev_performances_2025}. \NEW{While the \textit{Chebyshev Accelerated Subspace iteration Eigensolver} (ChASE) already permits to compute several thousands of the most extreme eigenpairs of hermitian matrices, this paper presents} an expansion of ChASE to compute several thousands of \NEW{the smallest positive} eigenpairs of pseudo-hermitian Hamiltonians, with convergence and performance similar to the hermitian variant.

In the remaining of the paper, we give a concise state-of-the-art on eigensolvers for solving Hamiltonians arising from excitonic physics. Then, we introduce the standard workflow of ChASE in the hermitian case, and highlight each part that requires an extension to the pseudo-hermitian one. After deriving the numerical properties of the pseudo-hermitian Hamiltonians, we tackle the algorithmic upgrade for each independent component of ChASE. Finally, numerical tests are presented on large pseudo-hermitian matrices, including strong scaling experiments and parallel efficiency on massively parallel multi-GPU platforms. 
\subsection{State-of-the-art}
Studying the optical properties and the electronic structure of materials has a strategical significance. In the perspective of improving energy production and storage, one important knowledge are the energies of quasiparticles called excitons, which can be extracted by solving the BSE \cite{salpeter_relativistic_1951}, as in the Yambo code \cite{marini_yambo_2009}. This step \NEW{is} carried out by computing a few of all the eigenvalues of the Hamiltonian matrix. The coupling term $B$ is occasionally negligible, therefore the \textit{Tamm-Dancoff Approximation} (TDA) \cite{hirata1999time} reduces the pseudo-hermitian eigenvalue problem \eqref{eqn::form_of_eigenproblem} to an hermitian one of the form
\begin{equation}
	AX = X\Omega,
	\label{introduction::resonant_eigenvalue_problem}
\end{equation}
where $X$ denotes a set of eigenvectors of the hermitian resonant $A$ and $\Omega$ the diagonal matrix whose entries contain their associated eigenvalues. Here, the eigenvalues on the diagonal of $\Omega$ are real, and the columns of $X$ form an orthonormal set. 

When the full spectrum is sought after, we have to employ a direct eigensolver of $\Ocal(m^3)$ complexity. Direct eigensolvers for hermitian matrices have been extensively studied over the past decades, and several classical algorithms---including the QR algorithm, divide-and-conquer, and bisection with the inverse iteration---are now standard in the literature \cite{cuppen_divide_1980} and have been integrated into the LAPACK library \cite{anderson_lapack_1999}. 

As the Hamiltonian grows rapidly with the number of atoms in a system, standard eigensolvers face significant memory and computational challenges. Considering that big matrices cannot always fit on a shared memory platform\footnote{A complex double precision dense matrix of size $100000$ requires almost 150Gb.}, efficient parallel computations on modern large-scale distributed-memory platforms are essential.

There exist several high-performance libraries targeting the full spectrum of hermitian matrices on distributed-memory machines, including ScaLAPACK~\cite{blackford_scalapack_1997},\\ EigenEXA~\cite{imamura_2011_eigenexa_development}, SLATE~\cite{gates_evolution_2025}, and ELPA~\cite{marek_elpa_2014}. Recently, the cuSolverMp library introduced by NVIDIA as a production-level implementation of dense eigensolver on distributed-memory GPU platforms has demonstrated the capability to compute millions of eigenvectors for matrices of size exceeding $10^{6}$ on modern supercomputers \cite{tal_solving_2024}. \NEW{When only the smallest eigenpairs -- which correspond to the lowest energy levels of the system -- are of interest, computing the full spectrum is unnecessary}. Iterative methods of complexity $\Ocal(m^2)$ may provide a more suitable alternative. The SLEPc \cite{hernandez_slepc_2005} library incorporates implementations of several iterative algorithms based on Krylov procedures (e.g., Lanczos). 

Filtered subspace iterative methods, such as the one implemented in ChASE \cite{winkelmann_chase_2019}, have a different approach and target the computation of eigenvectors associated with the most extreme eigenvalues by applying a Chebyshev polynomial filter at each iteration. From the filtered subspace, a Rayleigh-Quotient of small size is constructed, and the Ritz pairs are generally computed by performing an eigendecomposition with dense algorithm (e.g., divide and conquer). Compared to traditional Krylov approaches, ChASE exhibits several advantages on modern architectures, in particular GPUs: its reliance on dense matrix–matrix multiplications (GEMMs) allows for high efficiency on accelerated hardware, its block formulation naturally improves data locality, and the algorithm is well suited to exploit parallelism at scale \cite{wu_advancing_2023}. 

While the TDA generally provides sufficiently accurate solutions to the BSE, it may become inaccurate or even physically inappropriate in certain cases \cite{gruning_exciton-plasmon_2009,olevano_excitonic_2001}. When the coupling term $B$ cannot be neglected, the full eigenproblem is written
\begin{equation}
    HV = V\Lambda,
    \label{introcution::full_eigenproblem}
\end{equation}
where $V$ is the set of right eigenvectors, each paired with an eigenvalue held by the diagonal matrix $\Lambda$. When condition \eqref{introduction::quasi_hermitian_condition} is satisfied, the eigenvalues $\Omega$ of matrix $A$ provide an upper bound for the eigenvalues  $\Lambda$ of the full Hamiltonian $H$~\cite{shao_structure_2016}. The TDA-generated shift between $\Omega$ and $\Lambda$ is generally non-negligible, indicating that the spectrum of $A$ alone may not be sufficient for accurate simulations. Moreover, because $H$ is twice as large as $A$ due to the inclusion of the coupling term $B$, the complexity of a suitable direct solver would presumably be eight times larger. The actual cost to solve the full eigenproblem depends on the choice of the specific algorithm implementation, but it would nonetheless be substantially higher than the hermitian case. The observations above motivate the development of scalable eigensolvers for pseudo-hermitian matrices, achieving performance comparable to that of hermitian solvers, which can provide significant improvements for large-scale materials science simulations.

Based on the exploitation of the pseudo-hermitian structure of $H$ in \eqref{introduction::hamiltonian_definition}, the authors of \cite{shao_structure_2016} proposed a direct method that solves the full eigenproblem \eqref{introcution::full_eigenproblem} by computing the Cholesky factorization of a real symmetric positive definite (SPD) matrix that is spectrally equivalent to $SH$. The Cholesky factors are then recycled to construct a hermitian matrix sharing the same eigenvalues as the pseudo-hermitian Hamiltonian $H$, while the eigenvectors are recovered through back transformations. This method presents two notable advantages. The spectral properties of the pseudo-hermitian matrix are preserved despite the backward error of the Cholesky factors, and that a standard direct hermitian eigensolver can be applied, bypassing the non-hermitian structure of $H$. A related approach was introduced in \cite{penke_high_2020}, where the eigendecomposition is instead carried out on a real skew-symmetric matrix, allowing the use of real arithmetic in the diagonalization step. This variant has been implemented in ELPA and has been shown to outperform the complex hermitian eigensolver in ScaLAPACK by a factor of $3$. A more extensive summary of direct eigensolvers for the BSE can be found in \cite{penke_efficient_2022}. 

On the side of iterative eigensolvers for pseudo-hermitian matrices, a straightforward option is to apply general non-hermitian methods such as bi-Lanczos or Arnoldi. These approaches are typically more expensive and fail to exploit the specific properties of $H$. \NEW{The eigenvalues of interest lie in the middle of the spectrum due positive-negative symmetry (see Theorem \ref{theorem::SH_HPD}). This scenario is particularly difficult for the convergence of Lanczos and other Krylov subspace methods. }To the best of our knowledge, the first Lanczos-based method tailored to the pseudo-hermitian Bethe-Salpeter Hamiltonian, with a computational cost comparable to the hermitian case, was introduced in \cite{gruning_implementation_2011}. A few years later, the stability of this algorithm was enhanced \cite{shao_structure_2018}, leading to improved convergence. Building upon these advances, three variants of the thick-restarted Lanczos method \cite{wu_thick-restart_2000} have been recently designed to compute a few hundred eigenpairs efficiently~\cite{alvarruiz_variants_2025, milev_performances_2025}. Numerical experiments show promising performance in terms of scalability and convergence. Nevertheless, the authors acknowledge that extending these methods to compute a much larger portion of the spectrum---potentially several thousand eigenpairs---is challenging (see \cite[Section~6]{alvarruiz_variants_2025}): \NEW{full reorthogonalization causes the computational cost to grow with the number of converged eigenvectors.}\OLD{As more eigenpairs are computed, restarted Lanczos methods tend to lose orthogonality and become numerically unstable.} The adoption of alternative strategies, such as subspace iteration, could lead to faster convergence. By using polynomial filtering techniques, which can selectively amplify the desired spectral intervals, such methods exploit high-performance computing through GEMMs and parallelism.

The goal of this work is to extend subspace iteration with polynomial filtering to solve pseudo-hermitian Hamiltonians with the goal of achieving fast and scalable convergence for the computation of several thousand \OLD{eigenpairs}\NEW{of the smallest positive eigenpairs that lie in the middle of the spectrum}. In Section \ref{sec::chase_hermtian_case}, we recall the classical workflow of subspace iteration as it is implemented in the ChASE library. Section \ref{sec::properties_of_quasihermitian} examines the spectral properties of the pseudo-hermitian Hamiltonian to identify the components that require modification. The corresponding algorithmic developments are addressed in Section~\ref{sec::chase_quasihermtian_case}, including parameter estimation in the setup phase\NEW{, the filtering with the Chebyshev polynomial,} and the construction of the Rayleigh–Quotient. In Section \ref{sec:rayleigh_ritz_variants_convergence}, we prove that our oblique variant of the Rayleigh-Ritz method enables quadratic convergence of the Ritz-values, mirroring the hermitian case. Section \ref{sec::benchmarks} reports convergence and strong scaling experiments performed on large GPU platforms and solving for pseudo-hermitian Hamiltonians of different sizes. These results demonstrate that ChASE can be successfully generalized to the pseudo-hermitian setting while preserving both efficiency and scalability.

\section{ChASE for hermitian matrices}
\label{sec::chase_hermtian_case}
 The standard ChASE algorithm computes the $\nev$ \NEW{most extreme} eigenvalues of the hermitian matrix $A$ in \eqref{introduction::resonant_eigenvalue_problem}, with their associated eigenvectors. \NEW{These extremes correspond either to the largest or to the smallest eigenvalues. In most physical applications, the eigenvalues of interest are the smallest ones, located at the extreme left end of the spectrum. The two main components of ChASE} are: (i) a well designed Chebyshev polynomial that filters the external eigenvectors, and (ii) a Rayleigh-Ritz procedure that extracts approximations of the eigenpairs from the orthonormalized filtered subspace. Let $\omega_i$ denote the $i$th smallest eigenvalue of $A$. While ChASE targets the interval $[\omega_1,\omega_{\nev}]$, the filter is constructed over the slightly enlarged interval  $[\omega_1,\omega_{\nev+\nex}]$ using a Chebyshev polynomial of degree $\deg$. This enlargement increases the rank of the search space to $\nevex \defeq \nev + \nex$, which improves the accuracy of the filter near $\omega_{\nev}$ and accelerates convergence. Conversely, increasing $\deg$ enhances the damping of unwanted eigenvectors by flattening the polynomial $p$ near zero on the complement interval $[\omega_{\nevex},\omega_m]$, where $\omega_{\nevex}~=~\omega_{\nev+\nex}$. The effectiveness of this approach depends on accurate estimates of the spectral bounds $\omega_1,\omega_{\nevex}$ and $\omega_m$, which are respectively approximated by the parameters $\mu_1$, $\mu_{\nevex}$ and $\mu_m$ obtained from a few Lanczos iterations (25 is the default value). The smallest parameter $\mu_1$ is taken as the smallest approximated eigenvalue from Lanczos iterations, and the largest parameter $\mu_m$ is an upper bound of the spectrum based on \cite{ZHOU2011480}. The $\mu_{\nevex}$ value is estimated by integrating an approximation of the spectral density function \cite{lin_approximating_2016}, so that the interval $[\mu_1,\mu_{\nevex}]$ covers $\nicefrac{\nevex}{m}$ of the spectral weight. The resulting polynomial filter is illustrated in Figure \ref{fig:chase_polynomial}.
\begin{figure}[h!]
    \centering\captionsetup{justification=centering}
    \includegraphics[scale=0.9]{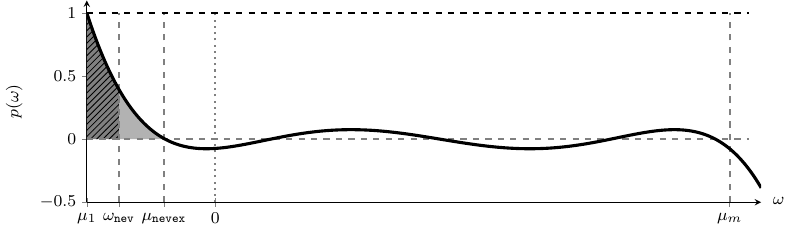}
    \caption{Chebsyhev polynomial filter - the gray filling corresponds to the filtered area and the dashed pattern represents the area of target eigenpairs.}
    \label{fig:chase_polynomial}
\end{figure}

The search space $W$ is initialized with random vectors drawn from a normal distribution. At each iteration, $W$ is \OLD{filtered and subsequently orthonormalized using either a distributed Cholesky QR or a block Householder QR algorithm, depending on its condition number.}\NEW{filtered via the three-term recurrence relation
\begin{equation}\label{eqn::three-term_recurrence}
    W_{i+1} \defeq \alpha_{i}AW_{i} + \beta_{i-1}W_{i-1}\quad \text{with}\quad i \leq \deg \quad \text{and} \quad W_1 \defeq AW_0,\quad W_0 \defeq W, 
\end{equation}
and where $\alpha_i$ and $\beta_{i-1}$ are scalars. The filtered space is subsequently orthonormalized using either a distributed Cholesky QR or a block Householder QR algorithm, depending on its condition number \cite{napoli_estimating_2026}.} In both cases, the orthonormalized vectors are collected in the $m \times \nevex$ matrix $Q$. ChASE proceeds with a Rayleigh-Ritz step to compute approximate eigenpairs: A Ritz pair $(\Tilde{\omega}, \Tilde{\bx})$ approximates an exact eigenpair $(\omega, \bx)$, where $\Tilde{\omega} \in \mathbb{R}$ and $\Tilde{\bx} \in \mathrm{range}(Q)$, with $\omega$ an eigenvalue in the diagonal matrix $\Omega$ of \eqref{introduction::resonant_eigenvalue_problem}, and $\bx$ an eigenvector from the set $X$. Each Ritz-pair should satisfy the following Galerkin condition (see e.g. \cite{saad_numerical_2011})
\begin{equation}\label{eqn::rayleigh-ritz_as_an_orthogonal_projection}
    \Tilde{\by}^*\left(A\Tilde{\bx} - \Tilde{\omega}\vphantom{\Tilde{\lambda}}\Tilde{\bx}\right) = 0, \quad \forall\Tilde{\by}\in \text{range}(Q).
\end{equation}

The orthonormal basis $Q$ defines the hermitian Rayleigh-Quotient $G \defeq Q^*AQ$ , from which the Ritz pairs are obtained by solving the reduced problem $GZ = Z\Tilde{\Omega}$. The Ritz values are given directly by the diagonal entries of $\Tilde{\Omega}$, while the corresponding Ritz vectors $\Tilde{X}$ are recovered as $\Tilde{X}~\defeq~QZ$, so that each $\Tilde{\bx}$ is a column of $\Tilde{X}$. In the hermitian case, Cauchy's interlace theorem guarantees that the eigenvalues of $G$ interlace with those of $A$, which facilitates locating the Ritz values $\Tilde{\Omega}$. An approximate eigenvector $\Tilde{\bx}$ can be decomposed as the linear combination $\Tilde{\bx} = \gamma \bx + \sigma \be$, where $\bx$ is the target eigenvector and $\be\perp\bx$. In the hermitian case, a key property established in \cite{stewart_matrix_2001} is that the orthogonal Rayleigh–Ritz procedure converges quadratically with respect to $\sigma$, in the sense that
\begin{equation}
    \big|\omega - \Tilde{\omega}\big| \leq \Ocal\big(\sigma^2\big).
    \label{eqn::hermitian_rayleigh-ritz_quadratic}
\end{equation}

This quadratic convergence is central to the efficiency of ChASE, as it ensures the extremely rapid convergence of Ritz values with sufficiently accurate filtered subspace. In Section \ref{sec:rayleigh_ritz_variants_convergence}, dedicated to Rayleigh-Ritz procedure, we demonstrate that this desirable convergence property can be replicated in the pseudo-hermitian case. Preserving this feature is a cornerstone for achieving fast convergence. The final step of each iteration then consists of computing the residuals $\br_i~\defeq~||A\Tilde{X}_{:,i}~-~\Tilde{\omega}_i\Tilde{X}_{:,i}||_2$ and locking each Ritz vector $\Tilde{X}_{:,i}$ that satisfies $\br_i \leq \tol$, where $\tol$ is a prescribed input tolerance. The number of locked vectors $\locked$ is incremented accordingly. The subspace iteration proceeds until all $\nev$ desired smallest eigenvalues are locked (i.e., $\locked = \nev$) or the prescribed maximal number of iterations is reached.

In what follows, we define $k \defeq \nevex - \locked$ as the number of non-converged eigenvectors. Each new iteration reuses the non-locked columns of $\Tilde{X}$ to form a reduced search space $W$ of size $k$, which is re-orthonormalized against the locked vectors to preserve orthogonality. The standard workflow for hermitian matrices is summarized in Algorithm \ref{algortihm::hermitian_chase}, while details of the locking strategy are discussed in \cite{winkelmann_chase_2019}.

\newlength{\dlen} \settowidth{\dlen}{{\sc Check for convergence}} 
\newlength{\ylen} \settowidth{\ylen}{Y}

\newlength{\clen}
\settowidth{\clen}{\small{\sc Deflation \& Locking} (Start)}

\renewcommand{\algorithmicrequire}{\textbf{Input:}}
\renewcommand{\algorithmicensure}{\textbf{Output:}}

\begin{algorithm}[h!t]
  \caption{Standard workflow of ChASE for hermitian matrices}
  \begin{algorithmic}[1]
  \Require hermitian matrix $A$ of size $m \times m$, $\nev$, and $\nex$
  \Ensure $\left(\Omega_{\nev},X_{\nev}\right)$ with
  $\Omega_{\nev} = \diag(\omega_1, \ldots,\omega_{\nev})$ and $X_{\nev} = \left[
    \bx_1, \ldots, \bx_{\nev}\right]$
  \State $W\gets {\rm randn}(m,\nevex)$
  \State Estimate $\omega_1$, $\omega_{\nevex}$ and $\omega_m$ with Lanczos \label{lst:line:lanczos} 
  \While{$\locked < \nev$}
  \State Filter the vectors, $W \gets p (A) W$ \label{lst:line:cheby} 
  \State Orthonormalize, $Q \gets \big[Y,W\big] R^{-1}$ \label{lst:line:VR-1} 
  \State Select non-locked space only, $Q \gets Q_{:,\locked:\nevex}$ \label{lst:line:VR-2} 
  \State Compute Rayleigh quotient, $G \gets Q^*AQ$ \label{lst:line:rrstarts} 
  \State Solve the reduced problem with \texttt{HEEVD}, $G Z = Z \tilde{\Omega}$\label{lst:line:rrreduced} 
  \State Compute $W \gets Q Z$ \label{lst:line:rrends}
  \State Compute the residuals, Res$\big(\Tilde{\Omega}, W\big)$ \label{lst:line:resids}
  \State Lock the converged eigenpairs, $(\Omega, \Tilde{X}, W) \gets \text{Locking}\big(\Tilde{\Omega}, W\big)$
    \State $\locked \gets \text{size}(\Tilde{X})$
  \EndWhile
  \State $\Omega_{\nev}, X_{\nev} \gets \Omega, \Tilde{X} $ 
\end{algorithmic}
\label{algortihm::hermitian_chase}
\end{algorithm}

\section{Properties of pseudo-hermitian Hamiltonians}
\label{sec::properties_of_quasihermitian} In this section, we recall the main properties of pseudo-hermitian Hamiltonians to identify which components of ChASE require modification. Let $\bu$ and $\bv$ denote respectively the left and right eigenvectors associated with the eigenvalue $\lambda$ of $H$, such that
\begin{equation}
	\bu^*H = \lambda\bu^*\qquad \text{and} \qquad H\bv = \lambda\bv.
	\label{eqn::general_eigenvalues_relation}
\end{equation}
The eigenvalues of a pseudo-hermitian Hamiltonian come in quadruplets\OLD{ (see \cite[Theorem~1]{benner_remarks_2018})}\NEW{}. Specifically, if $\lambda$ is an eigenvalue of $H$,  then $\{\lambda,\bar{\lambda},-\lambda$,$-\bar{\lambda}\}$ also belong to the spectrum, whose associated eigenvectors satisfy the following theorem.
\begin{theorem}\label{theorem:pseudo-hermitian_hamiltonian_eigenvalues_relation}
The quadruplet $\{\lambda,\bar{\lambda},-\lambda$,$-\bar{\lambda}\}$ belongs to the spectrum of $H$, with\begin{equation}\label{eqn::general_symmetry}
	HS\bu = \bar{\lambda}S\bu \qquad , \qquad HJ \bar{\bu} = -\lambda J\bar{\bu} \qquad , \qquad HK\bar{\bv} = -\bar{\lambda}K\bar{\bv},
\end{equation}
where $K$ and $J$ are defined as follows
\begin{equation}
	K \defeq
	\begin{bmatrix}
		0 & I\\
		I & 0
	\end{bmatrix} \qquad \text{and} \qquad 	J \defeq \begin{bmatrix}
		0 & I\\
		-I & 0
	\end{bmatrix}.
\end{equation}
\end{theorem}\NEW{
\begin{proof}
    See \cite[Theorem~1]{benner_remarks_2018}.
\end{proof}}
An important implication of Theorem \ref{theorem:pseudo-hermitian_hamiltonian_eigenvalues_relation} is that one can deduce the\OLD{ most extreme} negative \OLD{eigenpairs}\NEW{right eigenvectors} from the \OLD{most }positive \OLD{ones}\NEW{left ones}, and vice-versa. Therefore, computing half \OLD{ eigenpairs is sufficient to recover the entire spectrum}\NEW{of the spectrum is sufficient}. 

When the condition \eqref{introduction::quasi_hermitian_condition} holds, the eigenvalues of $H$ are real. In this case, we can establish a direct link between the left and right eigenvectors via the matrix $S$ defined in \eqref{introduction::pseudo_hermitian_condition}. 
\begin{theorem}\label{theorem::SH_HPD}
	If $SH$ is HPD, then the eigenvalues of $H$ are real, and the left and right eigenvectors are related by
	\begin{equation}\label{theorem::U_and_V_relation}
		\lambda = \bar{\lambda} \qquad \Leftrightarrow \qquad \bu = S\bv.
	\end{equation}
\end{theorem}
\begin{proof}
    Since $SH$ is HPD, it admits a Cholesky factorization $SH = LL^*$ with $L$ lower triangular factor. Observing that $SS = I$, the eigenvalue problem \eqref{eqn::general_eigenvalues_relation} can then be written as
	\begin{equation}
		SSH\bv~=~\lambda\bv,
	\end{equation}
which is equivalent to the hermitian eigenvalue problem
	\begin{equation}
		L^*SL\by = \lambda \by, \qquad \by = L^*\bv.
	\end{equation}
Hence, $\lambda$ is real, i.e., $\lambda = \bar{\lambda}$. Furthermore, from Theorem \ref{theorem:pseudo-hermitian_hamiltonian_eigenvalues_relation} we know that $HS\bu = \lambda S\bu$ which can be interpreted as $H\bv = \lambda\bv$ with $S\bu = \bv$.
\end{proof}
Theorem \ref{theorem::SH_HPD} has \OLD{two}\NEW{three} immediate consequences. First, the left eigenvectors are obtained from the right ones by the action of $S$. Since $S$ has the block-diagonal structure prescribed by \eqref{introduction::pseudo_hermitian_condition}, this amounts to flipping the sign of the lower block of each right eigenvector. Explicit computation of left eigenvectors is therefore unnecessary: they can be deduced at negligible cost from the right eigenvectors. \NEW{Second, the negative right eigenvectors can be directly deduced from the positive right ones, such that
\begin{equation}
    H\bv = \lambda\bv \qquad \Leftrightarrow \qquad   HK\bar{\bv} = -\lambda K\bar{\bv}. \label{eqn::positive-negative_right_eigenvectors}
\end{equation}}\NEW{Third}\OLD{Second}, all eigenvalues lie on the real axis. In \cite[Theorem 4]{shao_structure_2016}, the authors demonstrate that the eigenvalues of $A$ provides an upper bound for the positive eigenvalues of $H$. Then, the magnitude of the eigenvalues is bounded by the spectral radius of $A$, as stated in Lemma \ref{lemma::interval_of_real_eigenvalues}. 
\begin{lemma}\label{lemma::interval_of_real_eigenvalues}
	If $SH$ is HPD, then $\lambda$, the eigenvalues of $H$, satisfy
	\begin{equation}
		\lambda \in \left[-\rho(A),\rho(A)\right].
	\end{equation}
\end{lemma}
Moreover, since $A$ is a principal block of the HPD matrix $SH$, it is itself HPD. 
\begin{justify}
In what follows, let 
\begin{equation}
    U \defeq [\bu_1 , \ldots , \bu_n], \qquad V \defeq [\bv_1 , \ldots , \bv_n]
\end{equation}
denote respectively the set of left and right normalized eigenvectors, where each $\bu_i$ and $\bv_i$ is the left and right eigenvector corresponding to the $i$th eigenvalue $\lambda_i$. 
Since $H$ is non-hermitian, neither $U$ or $V$ is orthonormal. Instead, they satisfy a bi-orthogonality relation, that the scalar product $d_{ij} \defeq \bu_i^*\bv_j$ is zero only if $i \neq j$. Using the relation $U = SV$ introduced in Theorem \ref{theorem::SH_HPD}, we obtain 
\begin{equation}
    D \defeq V^*SV  \qquad \Rightarrow \qquad H = V\Lambda D^{-1}V^*S,
    \label{eqn::bi-orthogonality}
\end{equation}
where $D$ is a diagonal matrix containing the non-zero entries $d_{ii}$ on its diagonal. Hence, $D$ is always full rank and invertible. \NEW{Indicating the sets of eigenvectors associated with positive and negative eigenvalues respectively by $V_+$ and $V_-$, define $D_+ \defeq V_+^*SV_+$ and $D_- \defeq V_-^*SV_-$. From \eqref{eqn::positive-negative_right_eigenvectors}, we also observe that
\begin{equation}
     D_- \defeq V_-^*SV_- = V_+^*KSKV_+ = -V_+^*SV_+ = -D_+.
\label{eqn::symmetry_diagonals}
\end{equation}
 Naturally, we also have $V_+^*SV_- = 0$. These features, along with the positive-negative symmetry of spectrum, will help us assessing the quadratic convergence of the Ritz values in Section \ref{sec:rayleigh_ritz_variants_convergence}.}

To conclude this section, we show that the condition number of $H$ equals the condition number of $SH$. This information will be useful for further theoretical developments.
\begin{theorem}\label{theorem::condition_number_same}
    If SH is HPD,
    \begin{equation} 
        \cond(SH) = \cond(H).
    \end{equation}
\end{theorem}
\begin{proof}
    Let $H = F \Sigma P^*$ be the singular value decomposition of $H$. Then, from the definition of a pseudo-hermitian matrix \eqref{introduction::pseudo_hermitian_condition}, we have $H^*~=~SHS~=~SF\Sigma P^*S$. Subsequently, it follows that $F = SP$. 
    Substituting this into the singular value decomposition of $H$ gives
    \begin{equation} 
        SH = P \Sigma P^*,
    \end{equation} 
    which is the eigendecomposition of $SH$. Hence, the condition number of $H$ satisfies
    \begin{equation}
        \cond(SH) = \frac{\lambda_{\max}(SH)}{\lambda_{\min}(SH)} = \frac{\sigma_{\max}(H)}{\sigma_{\min}(H)} = \cond(H).
    \end{equation}
\end{proof}

In summary, when $SH$ is HPD, the algorithm can leverage on the following properties: 1) The eigenvalues of $H$ are \textit{real}, the same as the hermitian case, and appear as \textit{positive-negative pairs}. 2) The \textit{left eigenvectors follow directly from the right ones through the action of $S$}, which simply flips the sign of their lower half block. 3) The right eigenvectors themselves are not orthonormal, but instead \textit{form a bi-orthogonal system with respect to $S$}. \NEW{4) The negative and positive right eigenvectors relate via \eqref{eqn::positive-negative_right_eigenvectors}.} These properties form the foundation for extending the ChASE framework to pseudo-hermitian problems, as detailed in the next section.
\end{justify}

\section{Extending ChASE to pseudo-hermitian Hamiltonians}\label{sec::chase_quasihermtian_case} The hermicity of the matrix \NEW{and the locality of the target eigenvalues at one spectral extreme} are central assumptions in the current design of ChASE. To extend the framework \NEW{to solve for the smallest positive eigenpairs of }pseudo-hermitian eigenproblems such as \eqref{introcution::full_eigenproblem}, \NEW{five} major components of the algorithm must be revised. \textit{(i) Spectral bounds estimation.} The target eigenpairs are now located in the middle of the spectrum. In Section \ref{sec::chebyshev_filter}, we design the Chebyshev polynomial on $H^2$ (at Line~\ref{lst:line:cheby}) to fold the spectrum at zero, such that all eigenvalues become positive. \NEW{Thus, t}\OLD{T}he estimation of $\lambda_1, \lambda_{\nevex}$, and $\lambda_n$ (Line~\ref{lst:line:lanczos} in Algorithm~\ref{algortihm::hermitian_chase}) \NEW{should map to the spectrum of $H^2$. In addition, our approach }is no longer based on standard hermitian Lanczos algorithm. Instead, we employ a few iterations of the pseudo-hermitian Lanczos variant developed in \cite{gruning_implementation_2011}. Details are given in Section \ref{sec::estimating_variables}. \NEW{\textit{(ii) Chebyshev filter.} As a side effect of filtering with $H^2$, obtaining the $\nev$ smallest positive eigenpairs requires doubling the size of the search-space because the negative counterparts are filtered in to the same degree. Nevertheless, we exploit the positive-negative symmetry of \eqref{eqn::positive-negative_right_eigenvectors} to restrict filtering on the positive approximate eigenvectors only. Minor adjustments in the parallel matrix-matrix products are also required to account for the pseudo-hermitian structure.} \textit{(iii), orthonormalization and locking of right eigenvectors.} After filtering the target eigenvectors associated with smallest eigenvalues, ChASE extracts the locked eigenvectors from the current search space (Lines~\ref{lst:line:VR-1}–\ref{lst:line:VR-2} in Algorithm~\ref{algortihm::hermitian_chase}). In the pseudo-hermitian case, the right eigenvectors are not orthonormal, meaning that the target vectors may remain correlated. As a result, orthonormalizing the locked columns against the remaining search space can impair convergence. Remedies for this issue are discussed in Section~\ref{sec::S_orthonormalization}. \textit{(iv), Rayleigh-Ritz procedure.} Constructing the Rayleigh quotient $G$ from a single orthonormal basis (Line~\ref{lst:line:rrstarts}) is effective in the hermitian case but performs poorly in the pseudo-hermitian setting. Instead, we adopt an oblique Rayleigh–Ritz scheme (Section~\ref{sec::rayleigh_ritz_for_pseudo_hermitian}) \NEW{which preserves the characteristic positive-negative spectral symmetry of $H$}. \NEW{This feature is crucial to distinguish between positive and negative eigenvectors which appear symmetrically in the search-space}. Section \ref{sec:rayleigh_ritz_variants_convergence} will also show that \NEW{our variant} achieves quadratic convergence of the Ritz values and allows us to form $G$ directly without explicitly constructing the dual basis. \NEW{\textit{(v), Back-transformation and residuals.} The back-transformed eigenvectors are coupled through the $K$-conjugate relation \eqref{eqn::positive-negative_right_eigenvectors}. This enforces symmetric handling of positive and negative eigenpairs throughout the iterative process, allowing back-transformation and residual computations to be performed on only half the search space. \textit{(vi), Subspace Initialization.} We conclude this section by showing that the initial subspace should be constrained to the $S$-positive manifold to ensure the stability of Rayleigh-Ritz}.

\subsection{Estimating $\lambda_1$, $\lambda_{\nevex}$ and $\lambda_n$} \label{sec::estimating_variables}
The estimates \NEW{$\mu_1$, $\mu_{\nevex}$ and $\mu_n$ approximate certain eigenvalues of the initial matrix and} provide initial spectral bounds for the filtering and iteration steps. \NEW{In our case, we construct the Chebyshev polynomial filter on $H^2$ to target the eigenvalues closest to zero. Therefore we estimate $\mu_1$, $\mu_{\nevex}$ and $\mu_n$ for $H$, and design the polynomial on the squared estimates, such that $\mu^2_1$, $\mu^2_{\nevex}$ and $\mu^2_n$ correspond to estimates of squared eigenvalues of $H^2$}.

\NEW{The estimation of these parameters} relies on a few Lanczos iterations that must be adapted for the pseudo-hermitian case. One simple approach, proposed in \cite{gruning_implementation_2011}, is to flip the sign of the lower part of the vectors after the matrix-vector product to emulate the multiplication by the $S$ matrix in \eqref{introduction::pseudo_hermitian_condition}. \NEW{As in \cite{shao_structure_2018,alvarruiz_variants_2025}, this implementation leverages a Lanczos procedure that is defined in terms of $S$-inner products.} In addition, the eigenvalues of $H$ are real and occur in symmetric opposite (positive–negative) pairs with respect to the zero. To preserve the symmetry in the spectral density, we use an even number of Lanczos iteration. More advanced variants \cite{alvarruiz_variants_2025, milev_performances_2025} improve numerical stability, but since our Lanczos stage only generates rough eigenvalue estimates, we adopt the simpler method \cite{gruning_implementation_2011}. \OLD{Moreover, in some cases, it may be beneficial to slightly overestimate the magnitude of $\lambda_{\nevex}$ and $\lambda_n$ thereby widening the filtering interval and ensuring that undesired eigenvectors are sufficiently damped.}

\NEW{The parameters $\mu_1$ and $\mu_n$ are chosen respectively as the approximate eigenvalues of $H$ with the smallest and largest magnitude, while $\mu_{\nevex}$ is obtained by integrating the spectral density only over the positive axis.} During ChASE iterations, $\mu_{\nevex}$ is updated from the largest non-converged \NEW{positive} Ritz value. This gradually narrows the filtering interval, focusing the search space on the target portion of the spectrum. \OLD{An alternative strategy could use a power method to estimate $\lambda_1$ and set $\mu_{\nevex} = 0.0$ initially, letting these values adjust during iterations. However, since the computational cost is dominated by the Chebyshev filter and the Rayleigh-Ritz procedure, investing slightly more in the initial Lanczos setup is preferable, as it improves the starting interval and can reduce the total number of iterations. This is our method of choice.}

\subsection{The Chebyshev Polynomial Filter}\label{sec::chebyshev_filter}
When the TDA applies, the lowest energy levels of the system correspond to the smallest eigenvalues of $A$. In the pseudo-hermitian case, however, the target portion of the spectrum correspond to the smallest positive eigenvalues, which lie near the center of the spectrum of $H$. Designing a new Chebyshev polynomial to filter out the largest positive eigenvalues and all the negative ones at the same time would severely hinder convergence. Instead, we implicitly square the matrix in the three-term recurrence relation \eqref{eqn::three-term_recurrence} to fold the spectrum at zero, thereby treating positive and negative eigenvalues equally. In the pseudo-hermitian case, the three-term recurrence relation becomes
\begin{equation}\label{eqn::three-term_recurrence_pseudo-hermitian}
    W_{i+1} \defeq \alpha_{i}H^2W_{i} + \beta_{i-1}W_{i-1}, \text{ with}\quad i \leq \deg \quad \text{and} \quad W_1 \defeq AW_0,\quad W_0 \defeq W.
\end{equation}
Since smallest negative and positive eigenvalues are filtered in to the same degree, \NEW{the search space will contain an approximation of both negative and positive eigenvectors. To capture the $\nevex$ smallest positive eigenvectors,} the size of the search space $W$ should be expanded to $2\times\nevex$. Denoting the search space by $W = [W_-,W_+]$, the two $n\times \nevex$ matrices $W_-$ and $W_+$ respectively correspond to the search spaces attributed to the $\nevex$ negative and $\nevex$ positive smallest eigenvectors. In practice, the size of $W$ is progressively reduced as eigenvectors converge with the number of remaining unconverged given by $k = 2\times (\nevex - \locked)$. Naively, this amounts to doubling the computational cost of the Chebyshev filter. The developments that follow demonstrate how to restrict filtering to $W_+$ while recovering the filtering for $W_-$ afterwards, thereby allowing a symmetric treatment of the two halves of $W$. Let $V_-$ and $V_+$ respectively be sets of negative and positive right eigenvectors, such that
\begin{equation}
    HV_{+} = V_{+}\Lambda_{+} \qquad \Leftrightarrow \qquad HV_{-} = V_{-}\Lambda_{-},
\end{equation}
where $\Lambda_{-}$ and $\Lambda_{+}$ are the diagonal matrices containing the associated eigenvalues on their entries. The next theorem shows that the action of the Chebyshev polynomial on $V_{-}$ relates with its application to $V_{+}$.
\begin{theorem}\label{theorem::polynomial_effect_eigenvectors}
    Let $p$ be a Chebyshev polynomial of $H^2$. Then,
    \begin{equation}
        p(H^2)V_{-} = K\overline{p(H^2)V_{+}}.
    \end{equation}
\end{theorem}
\begin{proof}
First, Equation \eqref{eqn::positive-negative_right_eigenvectors} relates the negative and positive eigenvectors such that $V_{-} = K\overline{V}_{+}$. Second, $p$ is even, which gives $p(\Lambda_{-}^2) = p([-\Lambda_{+}]^2) = p(\Lambda_{+}^2)$. Therefore, we have
\begin{equation}
    p(H^2)V_{-} = V_{-}p(\Lambda_{-}^2) = K\overline{V}_{+}p(\Lambda_{-}^2) = K\overline{V}_{+}p(\Lambda_{+}^2) = K\overline{p(H^2)V_{+}}.
\end{equation}
\end{proof}
The next theorem generalizes the latter result to the search space $W_{-}$, showing that the action of the polynomial on $W_{-}$ can be deduced from its action on $W_{+}$ if their relation replicates \eqref{eqn::positive-negative_right_eigenvectors}.
\begin{theorem}\label{theorem::relation_polynomial}
    If $W_{-}$ and $W_{+}$ are initially related such that $W_{-} = K\overline{W}_{+}$, then the action of the filtering on the approximate set of negative eigenvectors $W_-$ can be recovered from its application to $W_{+}$, such that
    \begin{equation}
        p(H^2)W_{-} = K\overline{p(H^2)W_{+}}.
    \end{equation}
\end{theorem}
\begin{proof}
    First, note that \eqref{eqn::positive-negative_right_eigenvectors} describes an involution such that $W_{-} = K\overline{W}_{+} \Leftrightarrow W_{+} = K\overline{W}_{-}$. Second, let $W_{-}$ and $W_{+}$ be expressed as linear combination of $V_{-}$ and~$V_{+}$,
    \begin{equation}\label{eqn::relation_vp_vm}
        W_{-} = V_{-}C_{-} + V_{+}C_{+} \quad \Leftrightarrow \quad W_{+} = K\overline{W}_{-} = V_{+}\overline{C}_{-} + V_{-}\overline{C}_{+}.
    \end{equation}
    Then, from Theorem \ref{theorem::polynomial_effect_eigenvectors}, applying the filter to $W_{-}$ leads to
    \begin{equation}
         p(H^2)W_{-} = K\left[\overline{p(H^2)V_{+}}C_{-} + \overline{p(H^2)V_{-}}C_{+}\right] = K\overline{p(H^2)W_{+}}.
    \end{equation}
    This proves the above statement.
\end{proof}
The theorem \ref{theorem::relation_polynomial} shows that the filtering can be restricted to only half the search space (by simply setting $W_0 = W_+$ in the three-term recurrence relation \eqref{eqn::three-term_recurrence_pseudo-hermitian}), the other half being recovered through the action of the $K$ matrix and a conjugate operation. Moreover, the next theorem demonstrates that the relation $W_{-} = K\overline{W}_{+}$ preserves a perfect symmetry in the distribution of negative and positive eigenvectors within the search space $W$.
\begin{theorem}\label{eqn::symmetric_contribution_range}
    Let $\Pi(W) \defeq W(W^*W)^{-1}W^*$ be the $l_2$-orthogonal projection operator onto the range of $W = [W_-,W_+]$, with $W_{-} = K\overline{W}_{+}$. Then, positive and negative eigenvectors contribute equally to the range of $W$ in the sense that
    \begin{equation}
        \forall \bv_+ \in V_+, \quad \|\Pi(W)\bv_+\|_2 = \|\Pi(W)\bv_-\|_2 \quad \text{with} \quad \bv_{-} = K\bar{\bv}_{+} \in V_-.
    \end{equation}
\end{theorem}
\begin{proof}
    First, the range of $W$ can be rewritten as
    \begin{equation}
        \text{range}(W) = \text{range}([W_+,W_-]) = \text{range}([K\overline{W}_-,K\overline{W}_+]) = \text{range}(K\overline{W}).
    \end{equation}
    Thus, developing the $l_2$-orthogonal projection operator onto the range of $W$ gives
    \begin{equation}
    \begin{aligned}
        \Pi(W) \bv_+ = \Pi(K\overline{W}) \bv_+ &= K \Pi(\overline{W}) K\bv_+ 
         &= K \Pi(\overline{W}) \bar{\bv}_-.
    \end{aligned}
    \end{equation}
    The $K$ matrix being unitary, the $l_2$-norm is preserved such that 
    \begin{equation}
        \|\Pi(W)\bv_+\|_2 = \|K \Pi(\overline{W})\bar{\bv}_-\|_2 =  \|\Pi(W)\bv_-\|_2.
    \end{equation}
\end{proof}
Theorem \ref{eqn::symmetric_contribution_range} shows that contribution of positive and negative eigenvectors in the range of the search space is symmetric. From Theorem \ref{theorem::relation_polynomial}, we observe that the symmetric contribution of the eigenvectors in $W$ is preserved through the action of the polynomial filter, i.e., the equality $p(H^2)W_{-} = K\overline{p(H^2)W_{+}}$ holds. We will see that our oblique variant of Rayleigh-Ritz preserves the same symmetry, thereby enabling an equal treatment of the approximate positive and negative eigenpairs throughout the entire workflow of ChASE.

\NEW{In addition to performing the filtering on $W_{+}$ only, one can leverage the pseudo-hermiticity of the Hamiltonian to avoid global communications similarly to the hermitian case. Indeed, the} filtering of the search space relies on a sequence of \texttt{GEMM} operations. When the search space is distributed within each column communicator of the 2D MPI grid, a straightforward parallel \texttt{GEMM} would require communicating the row blocks of the matrix to the processes associated with the same column index. In the current implementation of ChASE for hermitian problems, this communication overhead is avoided by leveraging the hermitian structure of the matrix, allowing each process to operate solely on its local data prior to a final reduction step. More precisely, the multiplication of each row of $A$ with the search subspace $Q$ is emulated by multiplying the transpose conjugate of the corresponding column blocks.

In the pseudo-hermitian setting, this approach cannot be applied directly. However, the defining relation \eqref{introduction::pseudo_hermitian_condition}, namely $H = SH^*S$, yields the equivalence
\begin{equation}
    HW_+ = \hat{W}_+ \qquad \Leftrightarrow \qquad H^*S\hat{W}_+ = S\hat{W}_+.
\end{equation}
This identity enables a communication-avoiding formulation of the filtering step in the pseudo-hermitian case. Specifically, the sign pattern encoded in $S$ is applied to the \NEW{first half of input search space $W_+$} prior to the matrix product, after which the conjugate transpose $H^*$ is multiplied locally. The desired result $\hat{W}_+$ is then recovered by re-multiplying $S$ to the intermediate output, since $\hat{W}_+ = S(S\hat{W}_+)$. In practice, all sign-flip operations corresponding to implicit multiplication by $S$, are performed in place, without additional data movement. At the end of the parallel matrix product, the allocated memory for $W_+$ temporarily contains $SW_+$ and is subsequently restored to its original signs, because $W_+$ is required thereafter. 

Compared to the hermitian case, the pseudo-hermitian formulation introduces three sign-flip operations on the lower portions of tall-and-skinny matrices: one applied to $W_+$ prior to the matrix product, one applied to the intermediate result $S\hat{W}_+$ to obtain $\hat{W}_+$, and one to restore $W_+$. The sign-flip operations are required only when the search space is distributed within the column communicators. Since the multiplication of the search space distributed within the column communicators returns a matrix distributed within the row communicators, these three sign-flip operations are applied every two multiplication with $H$ only. \NEW{Therefore, one multiplication with $H^2$ in \eqref{eqn::three-term_recurrence_pseudo-hermitian} requires three sign-flip operations on $W_+$.}

These sign-flip operations are computationally inexpensive and map efficiently to GPU architectures. Moreover, the costs of the matrix products and the flip sign operations on $W_+$ decrease over successive iterations, as the number of approximate converged eigenvector increases and the effective dimension of the search subspace $W$ shrinks correspondingly.

\subsection{Orthonormalization and locking}\label{sec::S_orthonormalization} 
In the orthogonal Rayleigh–Ritz variant, the orthonormality of the search space $Q$ is a requirement. Although we adopt the oblique variant for pseudo-hermitian problems, orthonormality remains important for constructing the dual basis. \NEW{After filtering, recall that the right half $W_{-}$ of the search space $W$ is recovered from $W_{+}$ via the relation \eqref{eqn::positive-negative_right_eigenvectors}. Including $W_{-}$ in the search space is essential to ensure that both positive and negative eigenvectors are equally represented in $W$. The subsequent Rayleigh-Ritz step then allows us to distinguish between them using the orthornomalized matrix $Q$.}

In ChASE for the hermitian case, the search space is not only orthonormalized, but also orthonormalized against all already converged eigenvectors. This improves convergence by removing the information that has already been processed. However, this strategy only applies to hermitian matrices, where eigenvectors are orthonormal. In the general non-hermitian case, removing converged vectors from the search space can be counterproductive: because the target vectors may correlate with the converged ones, such removal can prevent convergence and cause stagnation.

For the pseudo-hermitian matrices, Eq. \eqref{eqn::bi-orthogonality} indicates that the right eigenvectors $V$ are orthogonal to $SV$. Therefore, the converged eigenvectors, after flipping the sign of their lower parts, are irrelevant to the remaining non-converged vectors. To preserve convergence, the remaining search space should be orthonormalized against these sign-flipped locked eigenvectors. The only computational overhead compared to the hermitian orthonormalization step is one flip-sign operation performed at each iteration on the set of locked vectors.

\NEW{That said, the set of locked eigenvectors $Y$ should contain not only the converged positive eigenvectors, but also the negative ones. At each iteration, the new locked negative eigenvectors are obtained by applying the $K$ operator to the complex conjugate of the newly converged positive ones.} Line \ref{lst:line:VR-1} of Algorithm \ref{algortihm::hermitian_chase} becomes 
\begin{equation}
    Q \gets \big[SY_+,SY_-,W\big] R^{-1} \qquad \text{with} \qquad Y_- \defeq K\overline{Y}_+.
\end{equation} 
\subsection{The Rayleigh-Ritz Procedure}\label{sec::rayleigh_ritz_for_pseudo_hermitian} \NEW{As negative and positive eigenvectors contribute equally to the range of $  Q  $, the Rayleigh-Ritz step is what enables the separation between approximate negative and positive eigenvectors, thereby allowing us to track the desired smallest positive ones.}

For hermitian matrices, the orthogonal Rayleigh–Ritz variant guarantees quadratic convergence of Ritz values toward the exact eigenvalues. This guarantee is tied directly to the orthogonality of eigenvectors and the variational characterization of Ritz values. However, in the non-hermitian or pseudo-hermitian setting, these properties break down, and quadratic convergence is no longer ensured. This limitation motivates the \NEW{development} of an oblique variant \NEW{that preserves the positive-negative symmetry that characterizes the eigenpairs and with quadratic convergence of the Ritz values. We hereby present an oblique variant that satisfies both conditions}.

\NEW{The oblique variant of Rayleigh-Ritz} employs a dual basis $Q_{\Lcal}$ satisfying the bi-orthogonality relation $Q_{\Lcal}^*Q=I_k$ to construct a Rayleigh-Quotient of the form $G \defeq Q_{\Lcal}HQ$. The subscript ``$\Lcal$'' emphasizes that the range of $Q_{\Lcal}$ should approximate the left invariant subspace. In what follows, we show that choosing
\begin{equation}
    Q_{\Lcal} \defeq SQ\big(Q^*SQ\big)^{-1}
    \label{eqn::left_basis_hermitian_rr}
\end{equation}
leads to an oblique Rayleigh-Ritz which solves a hermitian eigenvalue problem \NEW{with positive-negative spectral symmetry}. \NEW{Yet, we assume that $Q^*SQ$ is invertible, and discuss its eventual singularity thereafter.} By construction, we have $Q_{\Lcal} ^*Q = I_k$, ensuring that the bi-orthogonality property required for Rayleigh-Ritz is satisfied. The associated Rayleigh-Quotient then reads
\begin{equation}
    G = \big(Q^*SQ\big)^{-1}Q^*SHQ.
    \label{eqn::non-hermitian_rayleigh-quotient_heevd}
\end{equation}

In general, \NEW{performing the eigendecomposition of $G$ would require a non-hermitian direct eigensolver. The next theorem shows that our oblique variant admits a spectral equivalence with a hermitian eigenvalue problem. In addition to providing a real spectrum, it will allow us using a hermitian direct eigensolver}.

\begin{theorem}\label{theorem::spectral_equivalent_hermitian_eigenvalue problem}
     Let $Q^*SHQ = LL^*$ be the Cholesky factorization, with $L$ the lower triangular factor. Then, the Ritz values $\Tilde{\lambda}$ and the corresponding right and left Ritz vectors $\bz$ and $\bw$ satisfy
    \begin{equation}\label{eqn::spectral_equivalent_hermitian_eigenvalue problem}
        L^*\big(Q^*SQ\big)^{-1}L \by = \Tilde{\lambda}\by, \quad \text{with} \quad \by = L^*\bz \quad \text{and} \quad \bw = L\by.
    \end{equation}
\end{theorem}
\begin{proof}
   Define $\by = L^*\bz$. Then, we have
    \begin{equation}
        G \bz = \big(Q^*SQ\big)^{-1}Q^*SHQ\bz = \tilde{\lambda}\bz \quad \Leftrightarrow \quad \big(Q^*SQ\big)^{-1}L\by = \Tilde{\lambda} L^{-*}\by.
    \end{equation}
    Left multiplying by $L^{*}$ leads to the hermitian eigenvalue problem \eqref{eqn::spectral_equivalent_hermitian_eigenvalue problem}, which defines $\by$. Similarly, by taking the transpose-conjugate and multiplying from the right with $L^{-*}$, the left Ritz vector is $\bw=L\by$.
\end{proof}

In agreement with the spectral properties of $H$ described in Section \ref{sec::properties_of_quasihermitian}, the spectral equivalent of $G$ with an hermitian matrix implies that the approximate eigenvalues are also real. The hermitian spectrally equivalent Rayleigh-Quotient is built without explicit construction of the dual space $Q_{\Lcal}$. The Rayleigh-Ritz method is summarized in Algorithm \ref{algortihm::pseudo-hermitian_rayleigh_quotient_construction}, where the right comments contain the name of the basic algebra kernels along with an estimation of their computational complexity. Since the index $k$ represents the number of columns of $Q$, it initially corresponds to $2\cdot\nevex$ but decreases with the number of converged locked eigenvectors (i.e., \NEW{$k = 2\cdot (\nevex - \locked)$}). Clearly, the computational cost is dominated by the first \texttt{GEMM} operation at line 1 since $ k \ll n$. \NEW{Note that the explicit inversion of $Q^*SQ$ in \eqref{eqn::spectral_equivalent_hermitian_eigenvalue problem} can be avoided by formulating the inverse eigenproblem. In practice, one constructs the inverted Rayleigh-Quotient by applying backward substitution with $L^{*}$ on the right, followed by forward substitution with $L$ on the left (see Line 7 in Algorithm \ref{algortihm::pseudo-hermitian_rayleigh_quotient_construction}).} \NEW{While never experienced in practice, one could select a different dual basis in case of singularity of  $Q^*SQ$. For instance $Q_{\Lcal} \defeq \left[SQ - Q \left(Q^*SQ - \text{diag}(Q^*SQ)\right)\right]\text{diag}^{-1}(Q^*SQ)$ is a valid choice whenever the diagonal entries of $Q^*SQ$ are nonzero. While this option does not take advantage of any spectral equivalence to an hermitian matrix and therefore require a non-hermitian direct eigensolver, it could serve as a fallback strategy to the eventual corner cases where $M \defeq Q^*SQ$ is singular.}
\begin{algorithm}[h!t]
  \caption{Sequential construction of the hermitian Rayleigh Quotient}
  \begin{algorithmic}[1]
	\State Compute $T \gets HQ$ \Comment{\texttt{BLAS GEMM} $n^2 k $}
	\State Flip the sign of the lower half $T \gets ST$ \Comment{\texttt{INTERNAL FLIP} $\frac{n}{2}k$}
	\State Compute $M_k \gets Q^*T$ \Comment{\texttt{BLAS GEMM} $nk^2$}
	\State Factorize $L_k \gets \textsc{Cholesky}(M_k)$ \Comment{\texttt{LLAPACK PORTF} $k^3$}
	\State Compute $M_k \gets I-2Q_{2}^*Q_{2} \quad ( = Q^*SQ)$ \Comment{\texttt{BLAS GEMM} $\frac{n^2}{4}k$}
	\State Copy $G_k \gets M_k$ \Comment{\texttt{BLAS COPY} $k^2$}
	\State Solve $G_k \gets L_k^{-1}(G_kL_k^{-*})$ \Comment{2 \texttt{LLAPACK TRSM} $2k^2$}
	\State Solve $G_kY_k = \Tilde{\Lambda}_kY_k$ \Comment{\texttt{LLAPACK HEEVD} $k^3$}
	\State Back-transform $\Tilde{\Lambda}_k,Y_k \gets \Tilde{\Lambda}_k^{-1},L_k^{-1}Y_k$ \Comment{\texttt{LLAPACK TRSM} $k^2$}
	\State Return $\Tilde{\Lambda},Y_k$
  \end{algorithmic}
\label{algortihm::pseudo-hermitian_rayleigh_quotient_construction}
\end{algorithm}

On top of its spectral equivalence with an hermitian eigenvalue problem, Theorem \ref{eqn::positive-negative-symmetry_oblique_RR} stipulates that our oblique variant preserves the spectral symmetry of $H$.
\begin{theorem}\label{eqn::positive-negative-symmetry_oblique_RR}
Let $W = QR$ be a decomposition of the search space, where the two halves of $W = [W_+,W_-]$ relate via $W_{-} = K\overline{W}_{+}$. Let $Z_+$ and $Z_-$ be the right eigenvectors of $G$ associated with the positive and negative eigenvalues $\tilde{\Lambda}_+$ and $\tilde{\Lambda}_-$, respectively. The spectrum of $G$ is real positive-negative symmetric, i.e.,
\begin{equation}\label{eqn::positive-negative-symmetry_oblique_RR_eq}
    \tilde{\Lambda}_- = -\tilde{\Lambda}_+,
\end{equation}
and the associated back-transformed approximate eigenvectors of $V_+$ and $V_-$ replicate the characteristic $K$-conjugate relation \eqref{eqn::positive-negative_right_eigenvectors}, such that
\begin{equation}\label{eqn::positive-negative-symmetry_oblique_RR_eq}
    QZ_- = K\overline{QZ_+}.
\end{equation}
\end{theorem}
\begin{proof}
    By construction, $G$ in \eqref{eqn::non-hermitian_rayleigh-quotient_heevd} is composed of two parts, that are
    \begin{equation}
        (Q^*SQ)^{-1} = R(W^*SW)^{-1}R^* \quad \text{and} \quad  Q^*SHQ = (R^{-1})^*W^*SHWR^{-1}.
    \end{equation}
    Therefore, computing the positive eigenpairs of $G$ is equivalent to solving the generalized eigenvalue problem of the form
    \begin{equation}\label{eqn::generalized_eigenvalue_problem}
        W^*SHW\Phi = W^*SW \Phi\tilde{\Lambda}_+ \quad \text{with} \quad Z_+ = R\Phi.
    \end{equation}
     Our target is to show that \eqref{eqn::generalized_eigenvalue_problem} can be transformed into a pseudo-hermitian generalized eigenvalue problem, hence proving the relations prescribed by Theorem \ref{eqn::positive-negative-symmetry_oblique_RR}. In this way, two first ingredients are the coefficients of the linear decomposition of $W_+$ and $W_-$ in \eqref{eqn::relation_vp_vm}, namely $C_+$ and $C_-$. Developing the first component $W^*SW$ gives
    \begin{equation*}
    \begin{aligned}
        W^*SW = \begin{bmatrix}
            C_+^*V_+^*SV_+C_+ + C_-^*V_-^*SV_-C_-& C_+^*V_+^*SV_+\overline{C}_- + C_-^*V_-^*SV_-\overline{C}_+ \\
            \overline{C}_-^*V_+SV_+C_+ +  \overline{C}_+^*V_-SV_-C_- & \overline{C}_-^*V_+SV_+\overline{C}_- + \overline{C}_+^*V_-SV_-\overline{C}_+
        \end{bmatrix}.
    \end{aligned}
    \end{equation*}
    Then, recall the definitions $D_+ \defeq V_+^*SV_+$ and $D_- \defeq V_-^*SV_-$ of \eqref{eqn::symmetry_diagonals}, where we observed that $D_- = -D_+$. Also, recall the $S$-orthogonal relation between negative and positive eigenvectors, that is $V_+^*SV_- = 0$. Plugging these definitions into the above equation, we obtain
    \begin{equation*}
    \begin{aligned}
        W^*SW = \begin{bmatrix}
                        C_+^*D_+C_+ - C_-^*D_+C_-& C_+^*D_+\overline{C}_- - C_-^*D_+\overline{C}_+ \\
                        \overline{C}_-^*D_+C_+ - \overline{C}_+^*D_+C_- & \overline{C}_-^*D_+\overline{C}_- - \overline{C}_+^*D_+\overline{C}_+
                \end{bmatrix} =
                \begin{bmatrix}
                        A_S & B_S \\
                        -\overline{B}_S & -\overline{A}_S
                \end{bmatrix},
    \end{aligned}
    \end{equation*}
    with $A_S \defeq C_+^*D_+C_+ - C_-^*D_+C_-$ and $B_S \defeq C_+^*D_+\overline{C}_- - C_-^*D_+\overline{C}_+$. The blocks $A_S$ and $B_S$ satisfy $A_S = A_S^*$ and $B_S = -B_S^T$ (e.q., $\overline{B}_S = -B_S^*$). For the second component $W^*SHW$, define the two blocks $A_H \defeq C_+^*D_+\Lambda_+C_+ + C_-^*D_+\Lambda_+C_-$ and $B_H \defeq C_+^*D_+\Lambda_+\overline{C}_- + C_-^*D_+\Lambda_+\overline{C}_+$. Similarly, its development leads to
    \begin{equation}\label{eqn::similar_SH_matrix}
        W^*SHW = \begin{bmatrix}
        A_H & B_H\\
        \overline{B}_H & \overline{A}_H
        \end{bmatrix}, \quad \text{with} \quad A_H = A_H^* \quad \text{and} \quad B_H = B_H^T.
    \end{equation}
    Then, let $S_k$ and $K_k$ respectively be the two  $k \times k$ counterparts of $S$ and $K$, i.e.,
    \begin{equation}
        S_k \defeq \begin{bmatrix}
                I_{k/2} & 0 \\
                0 & -I_{k/2}
        \end{bmatrix} \qquad \text{and} \qquad K_k \defeq \begin{bmatrix}
                0 & I_{k/2} \\
                I_{k/2} & 0
        \end{bmatrix}.
    \end{equation}
    In addition, define $H_k \defeq S_k W^*SHW$ and $M_k \defeq S_kW^*SW$. From \eqref{eqn::similar_SH_matrix}, $H_k$ is pseudo-hermitian, and therefore characterized by the same structural properties as introduced in Section \ref{sec::properties_of_quasihermitian}. By multiplying from the left both sides of \eqref{eqn::generalized_eigenvalue_problem}  by $S_k$, one gets the equivalent pseudo-hermitian generalized eigenvalue problem
    \begin{equation}
       H_k \Phi = M_k \Phi\tilde{\Lambda}_+.
    \end{equation}
    Due to its pseudo-hermitian nature, $H_k$ satisfies $H_k = -K_k\overline{H}_kK_k$. On the other hand, $M_k$ has a particular block structure that enables a similar relation of the form $M_k = K_k\overline{M}_kK_k$. Since we know from Theorem \ref{theorem::spectral_equivalent_hermitian_eigenvalue problem} that $G$ has a real spectrum, we subsequently have 
    \begin{equation}
        -K_k\overline{H}_kK_k \Phi = K_k\overline{M}_kK_k \Phi\tilde{\Lambda}_+ \quad \Leftrightarrow \quad H_k K_k \overline{\Phi} = M_k K_k\overline{\Phi}(-\tilde{\Lambda}_+).
    \end{equation}
    As a consequence, the spectrum of the oblique Rayleigh-Quotient $G$ is also positive-negative symmetric such that $\Lambda_- = -\Lambda_+$, with $Z_+ = R\Phi$ and $Z_- = RK_k\overline{\Phi}$ the associated sets of eigenvectors, proving the first statement of the theorem. Regarding back-transformed eigenvectors, we simply observe that
    \begin{equation*}
    \begin{aligned}
        QZ_- = QRK_k\overline{\Phi} &= [W_+,W_-]K_k\overline{\Phi}\\
        &= [W_-,W_+]\overline{\Phi} = [K\overline{W}_+,K\overline{W}_-]\overline{\Phi}= K\overline{[W_+,W_-]\Phi} = K\overline{QZ_+},
    \end{aligned}
    \end{equation*}
    which is consistent with the second statement. This conclude the proof.
\end{proof}

\subsection{Back-transformation and computation of the residuals} We conclude this section by describing the last step of the iterative process: computing the residuals for the approximate eigenpairs. Specifically, the residual norms of the positive and negative approximate eigenpairs are strictly equal -- a property that ensures the symmetric spectral treatment of ChASE throughout iterations. This equality follows directly from Theorem \ref{eqn::positive-negative-symmetry_oblique_RR_eq}, which establishes that the eigenvalues of the Rayleigh-Quotient appear in positive-symmetric pairs, and that the corresponding back-transformed eigenvectors also relate via the characteristic $K$-conjugate relation \eqref{eqn::positive-negative_right_eigenvectors}. As a result, both back-transformation and the residuals are computed on only half the search space.

\subsection{Subspace Initialization} The convergence and numerical robustness of the pseudo-hermitian ChASE algorithm are highly sensitive to the properties of the initial subspace. To ensure the stability of the proposed Rayleigh-Ritz procedure, the subspace must be constrained to the $S$-positive manifold, where $SH$ is HPD (\ref{introduction::quasi_hermitian_condition}).

Consider a random trial vector $\bv \in \mathbb{C}^{2m}$ partitioned into components $\bx, \by \in \mathbb{C}^m$ such that $\bv^T = [\bx^T, \by^T]$. For an eigenpair $(\lambda, \bv)$ of $H$ with $\lambda > 0$, pre-multiplying $H\bv = \lambda \bv$ by $\bv^* S$ yields $\bv^* SH \bv = \lambda (\bv^* S \bv) = \lambda \left( \|\bx\|_2^2 - \|\by\|_2^2 \right)$. Because $SH$ is HPD, the components of the random vector $v$ should satisfy $\|\bx\|_2 > \|\by\|_2$ for $\lambda > 0$, with the upper block $\bx$ dominating the lower block $\by$. Contrary to such a requirement, the typical random initialization based on a normal distribution leads to an expectation value of $E[\bv^*S\bv] \approx 0$, implying that the initial search subspace is statistically likely to contain ``S-neutral'' directions. This state is numerically catastrophic for the pseudo-hermitian ChASE as it can cause failures in the Rayleigh-Ritz procedure. 

To resolve this issue. we propose to bound the coupling ratio $\gamma := \|\by\|_2 / \|\bx\|_2$. Writing $H\bv = \lambda \bv$ in terms of its blocks
\begin{equation}
    (A-\lambda I)\bx = - B\by  \quad \text{and} \quad (\bar{A}+\lambda I)\by = - \bar{B}\bx,
\end{equation}
and bounding its norm appropriately enables us to characterize the coupling ratio $\gamma$ via the following sandwich bounds 
\begin{equation}
   \frac{\sigma_{\min}(A - \lambda I)}{\|B\|_2} \leq \gamma \leq \frac{\|B\|_2}{\sigma_{\min}(\bar{A} + \lambda I)}.
\end{equation}
By further approximating the denominator using the smallest positive eigenvalue $\lambda_1$ of $H$, we obtain a practical global upper bound for $\gamma$, such that $\gamma \leq \frac{\|B\|_2}{2 \lambda_1}$. While we could compute these bounds on the fly\footnote{The 2-norm of $B$ could be substituted with the cheaper Frobenious norm}, we chose the simpler empirical value $\gamma =10^{-3}$ consistent with the physical properties of the pseudo-hermitian Hamiltonians of all the systems we have tested. 
By this setting,
we ensure $\|\bx\|_2^2 \gg \|\by\|_2^2$, which pushes the initial $E[\bv^*S\bv]$ deep into the $S$-positive manifold. We emphasize that the stability bound $\gamma \le ||B||_2/ (2 \lambda_1$) is effectively invariant with respect to the system size $m$, implying that our heuristic choice of $\gamma$ remains valid across diverse materials and large-scale matrices; it is sufficiently small to satisfy the stability guardrail of the BSE Hamiltonian, yet large enough to break the statistical symmetry of random noise.

\NEW{The upgrade of ChASE for pseudo-hermitian Hamiltonians is summarized in Algorithm \ref{algortihm::pseudo-hermitian_chase}. In particular, it shows that in addition to filtering on $W_+$ only, the back-transformation for the right eigenvectors plus the computation of the residuals can be performed on the positive ones only, due to the spectral symmetry enabled by our oblique variant of Rayleigh-Ritz.}
\NEW{\begin{algorithm}[h!t]
  \caption{New workflow of ChASE for pseudo-hermitian Hamiltonians}
  \begin{algorithmic}[1]
  \Require pseudo-hermitian Hamiltonian $H$ of size $n \times n$, $\nev$, and $\nex$
  \Ensure $\left(\Lambda_{\nev},V_{\nev}\right)$ with
  $\Lambda_{\nev} = \diag(\lambda_1, \ldots,\lambda_{\nev})$ and $V_{\nev} = \left[
    \bv_1, \ldots, \bv_{\nev}\right]$
  \State $W_+\gets {\rm randn}(n,\nevex)$, $W_+[m+1:2m, :] \gets \gamma \cdot W_+[m+1:2m, :]$
  \State Estimate $\lambda_1$, $\lambda_{\nevex}$ and $\lambda_m$ with pseudo-hermitian Lanczos \cite{gruning_implementation_2011} \label{lst:line:lanczos} 
  \While{$\locked < \nev$}
  \State Filter one half of the vectors, $W_+ \gets p (H^2) W_+$ \label{lst:line:cheby}
  \State Recover the other half, $W_- \gets K\overline{W}_+$ \label{lst:line:cheby} 
  \State Orthonormalize, $Q \gets \big[SY_+,SY_-,W\big] R^{-1}$ \label{lst:line:VR-1} 
  \State Select non-locked space only, $Q \gets Q_{:,2\locked:2\nevex}$ \label{lst:line:VR-2} 
  \State Compute the Rayleigh-Quotient, $G \gets Q_{\Lcal}^*HQ$ \label{lst:line:rrstarts_pseudo} 
  \State Solve the reduced problem with \texttt{HEEVD}, $G Z = Z \tilde{\Lambda}$\label{lst:line:rrreduced_pseudo} 
  \State Back-transform only half, $W_+ \gets Q Z_+$ \label{lst:line:rrends}
  \State Compute half of the residuals, Res$\big(\Tilde{\Lambda}_+, W_+\big)$ \label{lst:line:resids}
  \State Lock the converged positive eigenpairs, $(\Lambda_+, V_+, Y_+) \gets \text{Locking}\big(\Tilde{\Lambda}_+, W_+\big)$
  \State Recover the converged negative eigenvectors, $Y_- \gets K\overline{Y}_+$ \label{lst:line:cheby} 
    \State $\locked \gets \text{size}(V_+)$
  \EndWhile
  \State $\Lambda_{\nev}, V_{\nev} \gets \Lambda_+, V_+ $ 
\end{algorithmic}
\label{algortihm::pseudo-hermitian_chase}
\end{algorithm}}

\NEW{Now that the upgrade of ChASE to tackle the smallest positive eigenpairs of pseudo-hermitian Hamiltonians have been introduced, the next section \ref{sec:rayleigh_ritz_variants_convergence} shows that our oblique variant for pseudo-hermitian Hamiltonians produces a tighter control over perturbation sensitivity and creates a pathway for recovering quadratic convergence.}

\OLD{Depending on how the reduced problem is formulated, two main variants are introduced in this paper for pseudo-hermitian problems: 1) \textit{Rayleigh-Quotient with hermitian spectral equivalence}, leads to solve a reduced hermitian eigenproblem. This allows the use of stable hermitian eigensolvers (e.g., \texttt{HEEVD}), and preserves many of the favorable spectral properties of $H$. In particular, this variant produces a tighter control over perturbation sensitivity and creates a pathway for recovering quadratic convergence under stable conditions. While very unlikely in practice, corner cases of failure may happen with this variant of Rayleigh-Ritz. This motivates the implementation of the a more general variant based on a non-hermitian Rayleigh-Quotient. 2) \textit{Rayleigh-Quotient with no hermitian spectral equivalence}, leads to solve a general non-symmetric eigensolver (e.g., \texttt{GEEV}), that is also made available in ChASE for the eventual cases where $SH$ is indefinite. While being naturally more expensive than solving an hermitian eigenvalue problem, this variant can be used as a back-up in case of failure of the hermitian Rayleigh-Quotient.}
\OLD{Section \ref{sec:rayleigh_ritz_variants_convergence} develops rigorous error bounds and convergence criteria for the hermitian Rayleigh-Quotient. We analyze how the convergence order — linear versus quadratic — depends on the interplay between the oblique trial subspaces, the conditioning of the dual basis, and key quantities such as $\Tilde{\delta}$, $\sigma_{\bu}$, and $\sigma_{\bv}$.}

\section{Quadratic Convergence of the Oblique Rayleigh-Ritz}\label{sec:rayleigh_ritz_variants_convergence}
\NEW{Our structure preserving Rayleigh-Ritz variant with hermitian spectral equivalence rests} on the oblique principle. \rw{Therefore, its} performance depends critically on how the dual basis $Q_{\Lcal}$ interacts with the search space $Q$. \NEW{In this section, we recall general considerations on oblique Rayleigh-Ritz, and demonstrates the quadratic convergence of the Ritz values enabled by our particular choice of $Q_{\Lcal}$ in \eqref{eqn::left_basis_hermitian_rr}}.
\subsection{General considerations on Oblique Rayleigh-Ritz}
Let $(\Tilde{\lambda},\Tilde{\bu},\Tilde{\bv})$ denote an approximate eigen-triplet associated with the exact eigencomponents $(\lambda,\bu,\bv)$ corresponding to the eigenvalue, the left-eigenvector and right-eigenvector, respectively. In practice, left eigenvectors are not explicitly computed during the iteration, since they can be recovered \textit{a posteriori} from the right eigenvectors via the relation \eqref{theorem::U_and_V_relation}. That said, the left eigenvectors, like the right ones, play a crucial role in the analysis of convergence. 

We assume that all the eigenvectors along with their approximations are normalized, such that
\begin{equation}
    ||\bu||_2 = ||\bv||_2 = ||\Tilde{\bu}||_2  = ||\Tilde{\bv}||_2 = 1, \quad \delta \defeq \bu^*\bv, \quad \Tilde{\delta} \defeq \Tilde{\bu}^*\Tilde{\bv}.
    \label{eqn::vector_normalized}
\end{equation}
Additionally, both $\Tilde{\bu}$ and $\Tilde{\bv}$ can respectively be decomposed as
\begin{equation}
    \Tilde{\bu} = \gamma_{\bu} \bu + \sigma_{\bu} \be_{\bu}, \quad \Tilde{\bv} = \gamma_{\bv} \bv + \sigma_{\bv} \be_{\bv} \quad \text{with} \quad \be_{\bu} \perp \bu \quad \text{and} \quad \be_{\bv} \perp \bv.
    \label{eqn::vector_decomposed}
\end{equation}
In oblique variant, approximate eigenpairs satisfy the \textit{Petrov–Galerkin} condition \begin{equation}\label{eqn::rayleigh-ritz_as_a_projection}
    \Tilde{\bu}^*\left(H\Tilde{\bv} - \Tilde{\lambda}\Tilde{\bv}\right) = 0, \quad \text{with} \quad \Tilde{\bv} \in \text{range}(Q) \quad \text{and} \quad \forall\Tilde{\bu} \in \text{range}(Q_{\Lcal}),
\end{equation}
which explicitly enforces that the residuals are orthogonal to the dual subspace (see e.g., \cite{saad_numerical_2011}). \NEW{The Rayleigh-Quotient is constructed via $G\defeq Q_{\Lcal}^*HQ$.} Although we do not compute the left eigenvectors in practice, they can in principle be obtained by back-transformations \NEW{from} the eigendecomposition of $G$
\begin{equation}\label{eqn::solving_the_rayleigh_quotient}
    \Tilde{\bv} = Q\bz \quad \text{with} \quad G\bz = \Tilde{\lambda}\bz, \quad \text{and} \quad \Tilde{\bu} = Q_{\Lcal} \bw \quad \text{with} \quad \bw^*G = \Tilde{\lambda}\bw^*.
\end{equation}
Since $Q_{\Lcal}^*Q = I_k$ and $\|Q\bz\|_2 = 1$ because $Q$ is unitary (i.e., its columns form an orthonormal basis of the search subspace), we have
\begin{equation}\label{eqn::deriving_delta_tilde}
    \Tilde{\delta} \defeq \Tilde{\bu}^*\Tilde{\bv}=\frac{\bw^*Q_{\Lcal}^*Q\bz}{\|Q_{\Lcal}\bw\|_2\|Q\bz\|_2}=\frac{\bw^*\bz}{\|Q_{\Lcal}\bw\|_2}.
\end{equation}
To assess convergence, we first recall the role of the dual basis and derive bounds for the Ritz values in terms of $\Tilde{\delta}$ defined in \eqref{eqn::vector_normalized} and the error coefficients $\sigma_{\bu}, \sigma_{\bv}$ from \eqref{eqn::vector_decomposed}. These bounds will serve as the foundation for analyzing conditions under which the method achieves quadratic convergence.

\begin{theorem}\label{theorem::convergence_ritz_values_oblique_variant}
    The convergence of a Ritz value $\Tilde{\lambda}$ in the oblique variant satisfies
    \begin{equation}\label{eqn::convergence_ritz_values_oblique_variant}
        \big| \lambda - \Tilde{\lambda} \big| \leq \big|\Tilde{\delta}\big|^{-1}\, \big|\bar{\sigma}_{\bu}\sigma_{\bv}\big| \,\big\|H-\lambda I\big\|_2.
    \end{equation}
\end{theorem}
\begin{proof}
    Using \eqref{eqn::solving_the_rayleigh_quotient} and \eqref{eqn::deriving_delta_tilde} we can write
    \begin{equation}\label{eqn::why_lambda_tilde}
        \Tilde{\bu}^*H\Tilde{\bv} = \frac{\bw^*Q_{\Lcal}HQ\bz}{\|Q_{\Lcal}\bw\|_2} = \frac{\bw^*G\bz}{\|Q_{\Lcal}\bw\|_2} = \Tilde{\lambda}\frac{\bw^*\bz}{\|Q_{\Lcal}\bw\|_2} = \Tilde{\lambda}\Tilde{\delta},
    \end{equation}
    which allows us to express the approximate eigenvalue $\Tilde{\lambda} = \Tilde{\delta}^{-1}\Tilde{\bu}H\Tilde{\bv}$.
    Plugging the latter expression in the linear decompositions of \eqref{eqn::vector_decomposed} gives
    \begin{equation}\label{eqn::ritz_values_convergence_heevd_part1}
    \begin{split}
        \big| \lambda - \Tilde{\lambda} \big| & = \big| \lambda - \Tilde{\delta}^{-1}\,\Tilde{\bu}^*H\Tilde{\bv} \big|\\
        & = \big| \lambda - \Tilde{\delta}^{-1}\big[ \lambda  \big(\bar{\gamma}_{\bu}\gamma_{\bv}\bu^*\bv + \bar{\gamma}_{\bu}\sigma_{\bv}\bu^*\be_{\bv} + \bar{\sigma}_{\bu}\gamma_{\bv}\be_{\bu}^*\bv \big) + \bar{\sigma}_{\bu}\sigma_{\bv}\be_{\bu}^*H\be_{\bv}\big]\big|.
    \end{split}
    \end{equation}
    Expanding the scalar product $\Tilde{\delta} \defeq \Tilde{\bu}^*\Tilde{\bv}$ with \eqref{eqn::vector_decomposed} leads to 
    \begin{equation} \label{eqn::development_on_tilde}
        \Tilde{\delta} - \bar{\sigma}_{\bu}\sigma_{\bv}\be_{\bu}^*\be_{\bv} = \bar{\gamma}_{\bu}\gamma_{\bv}\bu^*\bv + \bar{\gamma}_{\bu}\sigma_{\bv}\bu^*\be_{\bv} + \bar{\sigma}_{\bu}\gamma_{\bv}\be_{\bu}^*\bv.
    \end{equation}
    Plugging \eqref{eqn::development_on_tilde} back into \eqref{eqn::ritz_values_convergence_heevd_part1} results in
    \begin{equation}\label{eqn::ritz_values_convergence_heevd}
    \begin{split}
        \big| \lambda - \Tilde{\lambda} \big| & = \big| \lambda - \Tilde{\delta}^{-1}\big[\lambda\big(\Tilde{\delta} - \bar{\sigma}_{\bu}\sigma_{\bv}\be_{\bu}^*\be_{\bv} ) - \bar{\sigma}_{\bu}\sigma_{\bv}\be_{\bu}^*H\be_{\bv}\big]\big|\\    
        & = \big|\Tilde{\delta}\big|^{-1}\cdot\big|\bar{\sigma}_{\bu}\sigma_{\bv}\be_{\bu}^*\big(H - \lambda I\big)\be_{\bv}\big|,
    \end{split}
    \end{equation}
    which provides an upper bound on the error and proves the statement of the theorem
    \begin{equation}
        \big| \lambda - \Tilde{\lambda} \big| \leq \big|\Tilde{\delta}\big|^{-1} \cdot \big|\bar{\sigma}_{\bu}\sigma_{\bv}\big| \cdot\big\|H-\lambda I\big\|_2.
    \end{equation}
    \NEW{A similar proof can also be found in \cite[Theorem~3.1]{hochstenbach_two-sided_2003}}.
\end{proof}

\begin{remark}[hermitian case]
If $H$ is hermitian, then left and right eigenvectors coincide: $\mathbf{u} = \mathbf{v}$. Consequently, $\tilde{\delta} = 1$ and $|\bar{\sigma}_{\mathbf{u}} \, \sigma_{\mathbf{v}}| = |\sigma_{\mathbf{v}}|^2$, and the bound reduces to
\[
|\lambda - \tilde{\lambda}| \leq |\sigma_{\mathbf{v}}|^2 \, \|H - \lambda I\|_2,
\]
recovering the classical quadratic convergence of the hermitian Rayleigh--Ritz of \eqref{eqn::hermitian_rayleigh-ritz_quadratic}.  

\end{remark}

\begin{remark}[non-hermitian case]
For non-hermitian or pseudo-hermitian $H$, left and right eigenvectors generally differ. Then $|\tilde{\delta}| < 1$ and $|\sigma_{\mathbf{u}}|$ may differ from $|\sigma_{\mathbf{v}}|$. Two key factors control convergence:
\begin{enumerate}
    \item \emph{Subspace projection errors:} The quantities $\sigma_{\mathbf{u}}$ and $\sigma_{\mathbf{v}}$ measure the misalignment of the search subspaces with the left and right invariant subspaces. Quadratic convergence occurs only if $|\sigma_{\mathbf{u}}|$ and $|\sigma_{\mathbf{v}}|$ are of the same order of magnitude.
    \item \emph{Dual subspace alignment:} The factor $|\tilde{\delta}|^{-1}$ amplifies the error if the approximate left and right subspaces are poorly aligned. Bounding $|\tilde{\delta}|^{-1}$ by a constant independent of $\sigma_{\mathbf{u}}$, $\sigma_{\mathbf{v}}$, and $\tilde{\lambda}$ improves convergence predictability. 
\end{enumerate}
\end{remark}
Unlike the hermitian case, the error in non-hermitian Rayleigh-Ritz depends both on how well the right eigenvectors are captured and how accurately the dual (left) subspace is approximated. The oblique procedure mitigates cross-term contamination by constructing a dual basis, but convergence is no longer automatically quadratic. Proper alignment of both subspace and control of $|\tilde{\delta}|$ are essential for fast convergence. With regards to the order of magnitude of $\sigma_{\bv}$ and $\sigma_{\bu}$, we note that
\begin{equation} \label{eqn::lower_bounds_for_sigmas}
    \big|\sigma_{\bv}\big| \geq \big\|\big(I-\Pi(Q)\big)\bv\big\|_2 \qquad \text{and} \qquad \big|\sigma_{\bu}\big| \geq \big\|\big(I-\Pi(Q_{\Lcal})\big)\bu\big\|_2,
\end{equation}
where $\Pi(Q) \defeq Q(Q^*Q)^{-1}Q^*$ is the $\ell_2$-orthogonal projection onto $\text{range}(Q)$ (similarly for $\Pi(Q_{\Lcal})$); see, e.g., \cite[section~4.3.1]{saad_numerical_2011}. In other words, $\sigma_{\bv}$ (resp., $\sigma_{\bu}$) measures the component of the exact eigenvector $\bv$ (resp., $\bu$) that lies outside the subspace spanned by $Q$ (resp., $Q_{\Lcal}$). These inequalities highlight that convergence depends critically on the alignment of the subspace $Q$ and $Q_{\Lcal}$ with the true right and left invariant subspaces. Importantly, even if $Q$ approximates $\bv$ well, there is no guarantee that $Q_{\Lcal}$ approximates $\bu$ equally well. When the magnitudes of $|\sigma_{\bu}|$ and $|\sigma_{\bv}|$ differ significantly, the product $|\sigma_{\bu}\sigma_{\bv}|$ in the error bound becomes dominated by the larger term, leading to only linear rather than quadratic convergence. 
Ideally, \NEW{the dual basis $Q_{\Lcal}$ should ensure that }the components of $\bu$ and $\bv$ outside their respective subspaces are of similar magnitude
\begin{equation}\label{eqn::ideal_lower_bounds}
    \left\|\left(I-\Pi(Q_{\Lcal})\right)\bu\right\|_2 = \left\|\left(I-\Pi(Q)\right)\bv\right\|_2,
\end{equation}
ensuring that the convergence of the Ritz values is as fast as possible. \NEW{We will see that this ideal equality holds with our choice of $Q_{\Lcal}$ established in \eqref{eqn::left_basis_hermitian_rr}.}

As shown in the next theorem, an upper bound for $|\Tilde{\delta}|^{-1}$ depends on the norm of the dual basis $Q_{\Lcal}$ and the eigenvalue condition number of the Rayleigh-Quotient $G$. In what follows, we indicate the eigenvalue condition number (see \cite[chapter 2]{wilkinson_algebraic_eigenvalue_1988} \NEW{or \cite[chapter 3]{saad_numerical_2011}}) of the Rayleigh-Quotient $G$ with $\cond(\Tilde{\lambda},G)~\defeq~(|\bw^*\bz|)^{-1}$, where $\bw$ and $\bz$ are respectively a left and right eigenvector as defined in \eqref{eqn::solving_the_rayleigh_quotient}. Similar to the usual interpretation of the condition number of a matrix, the eigenvalue condition number relates the perturbation of an eigenvalue with the perturbation of the associated matrix: when $\cond(\Tilde{\lambda},G)$ is large, a small perturbation on $G$ can produce a large perturbation of $\Tilde{\lambda}$. This quantity gives an indication on the numerical stability of the eigensolver.
\begin{lemma} \label{theorem::bound_for_delta}
    Let $\cond(\Tilde{\lambda},G)$ be an eigenvalue condition number of the oblique Rayleigh-Quotient $G$. Then, the oblique overlap $\big|\Tilde{\delta}\big|^{-1}$ is bounded as follows
    \begin{equation}\label{eqn::bound_for_delta}
        \big|\Tilde{\delta}\big|^{-1} \leq\; \cond(\Tilde{\lambda},G)\cdot\big\|Q_{\Lcal}\big\|_2.
    \end{equation}
\end{lemma}
\begin{proof}
The proof is a direct consequence of \eqref{eqn::deriving_delta_tilde}. Since the eigenvalue condition number is defined as the inverse of the angle between the associated left and right eigenvectors $\cond(\Tilde{\lambda},G)~\defeq~(|\bw^*\bz|)^{-1}$, we simply have 
\begin{equation}
    \big|\Tilde{\delta}\big| = \frac{|\bw^*\bz|}{\|Q_{\Lcal}\bw\|_2} = \frac{\cond(\Tilde{\lambda},G)^{-1}}{\|Q_{\Lcal}\bw\|_2} \quad \Rightarrow \quad \big|\Tilde{\delta}\big|^{-1} \leq  \cond(\Tilde{\lambda},G)\cdot\big\|Q_{\Lcal}\big\|_2.
\end{equation}
\end{proof}
In attempting to decrease the upper bound of $|\Tilde{\delta}|^{-1}$ in \eqref{eqn::bound_for_delta}, one could in principle set $Q_{\Lcal} \defeq Q$, because $Q^*Q = I_k$. However, this choice does not control $\cond(\tilde{\lambda},G)$: the reduced problem may still be ill-conditioned, and $\cond(\tilde{\lambda},G)$ can dominate the bound. Thus $Q_{\Lcal}=Q$ is not a panacea, and we have
\begin{equation*}
\begin{aligned}
    \left\|\left(I-\Pi(Q)\right)\bu\right\|_2 &= \left\|\left(S-QQ^*S\right)\bv \right\|_2 \\
    &= \left\|\left(I-\Pi(Q)\right)\bv + 2M\bv\right\|_2 \quad \text{with} \quad M \defeq \begin{bmatrix}
        0 & Q_1Q_2^*\\
        0 & Q_2Q_2^* - I
    \end{bmatrix},
\end{aligned}
\end{equation*}
where $Q_1$ and $Q_2$ respectively stand for the upper and lower part of $Q = [Q_1^T,Q_2^T]^T$. This extra $2M\bv$ stems from the pseudo-hermitian sign-flip relation between left and right eigenvectors and can be non-negligible, so even a small right-project error does not guarantee a small left-projection error. This explains why $|\sigma_{\bv}|$ and $|\sigma_{\bu}|$ can differ substantially and motivates designing $Q_{\Lcal}$ so that it directly approximates the left invariant subspace rather than relying on $Q_{\Lcal}=Q$.

\subsection{Rayleigh-Quotient with hermitian spectral equivalence}\label{subsec:rayleigh-quotient-hermitian-equivalence} 
\NEW{We now show that the two necessary conditions for quadratic convergence are satisfied by our oblique variant of Rayleigh-Ritz. Specifically, the $\ell_2$-projection operators $\Pi(Q_{\Lcal})$ and $\Pi(Q)$ yield the same lower bounds in \eqref{eqn::lower_bounds_for_sigmas} when applied respectively to left and right eigenvectors, so that the ideal equality \eqref{eqn::ideal_lower_bounds} holds. Moreover, we prove that $Q_{\Lcal}$ directly influences a constant bound for $|\Tilde{\delta}|^{-1}$, enabling the non-hermitian Rayleigh–Ritz procedure to achieve the same quadratic convergence rate as in the hermitian case. First,} the following theorem establishes that the ideal relation in \eqref{eqn::ideal_lower_bounds} holds in with the setting above.

\begin{theorem}\label{theorem::ideal_norm_equality_reached}
    Let $Q_{\Lcal}$ be defined as in \eqref{eqn::Q_LQ_L_definition_general_form}, namely $Q_{\Lcal} = SQ(Q^*SQ)^{-1}$. Then, the $\ell_2$-orthogonal projection operators $\Pi(Q)$ and $\Pi(Q_{\Lcal})$ satisfy
    \begin{equation}\label{eqn::ideal_norm_equality_reached}
        \left\|\left(I-\Pi(Q_{\Lcal})\right)\bu\right\|_2 = \left\|\left(I-\Pi(Q)\right)\bv\right\|_2.
    \end{equation}
\end{theorem}
\begin{proof}
    Developing the $\ell_2$-projection $\Pi(Q_{\Lcal})$ gives
    \begin{equation}
    \begin{split}
        \Pi(Q_{\Lcal}) & = Q_{\Lcal}\left(Q_{\Lcal}^*Q_{\Lcal}\right)^{-1}Q_{\Lcal}^*\\
        & = SQ\left(Q^*SQ\right)^{-1}\left(Q^*SQ\right)^{2}\left(Q^*SQ\right)^{-1} Q^*S = SQQ^*S.
    \end{split}
    \end{equation}
    Next, recall that $\bu = S\bv$. Since $Q$ has orthonormal columns, we obtain
    \begin{equation}
        \left(I-\Pi(Q_{\Lcal})\right)\bu = S\bv-SQQ^*\bv = S\left(I-\Pi(Q)\right)\bv.
    \end{equation}
    Finally, because $S$ is unitary, it preserves the $\ell_2$-norm. Hence 
    \begin{equation}
    \left\|\left(I-\Pi(Q_{\Lcal})\right)\bu\right\|_2~=~ \|S\left(I-\Pi(Q)\right)\bv\|_2~=~\|\left(I-\Pi(Q)\right)\bv\|_2,
    \end{equation}
    proving \eqref{eqn::ideal_norm_equality_reached}.
\end{proof}
From this result, we conclude that $Q_{\Lcal}$ and $Q$ approximate $\bu$ and $\bv$ with identical accuracy. Consequently,
$\Ocal(\sigma_{\bu})~\approx~\Ocal(\sigma_{\bv})$, which implies $|\bar{\sigma}_{\bu}\sigma_{\bv}| = \Ocal({\sigma_{\bv}^2})$.

In what follows, let $\lambda_{\min}(\cdot)$ and $\lambda_{\max}(\cdot)$ respectively be the minimum and maximum eigenvalue in magnitude of the designated matrix between parenthesis. To complete the proof of quadratic convergence of the Ritz values, it remains to establish an upper bound for $|\Tilde{\delta}|^{-1}$ that is independent of  both $\sigma_{\bv}$ and of the Ritz value $\Tilde{\lambda}$. This bound is derived in the next theorem, which refines Theorem \ref{theorem::bound_for_delta}.
\begin{theorem}
    The parameter $|\Tilde{\delta}|^{-1}$ is bounded as follows
    \begin{equation}
        \big|\Tilde{\delta}\big|^{-1} \leq \dfrac{\sqrt{\cond\big(H\big)}}{\left|\lambda_{\min}\big(Q^*SQ\big)\right|}. 
    \end{equation}
\end{theorem}
\begin{proof}
    From Theorem \ref{theorem::bound_for_delta}, we need to control $\|Q_{\Lcal}\|_2$ and the eigenvalue condition number $\cond\big(\Tilde{\lambda},G\big)$. First, recall that
    \begin{equation}
        \big\|Q_{\Lcal}\big\|^2_2 = \sigma_{\max}\big(Q_{\Lcal}\big)^2 = \lambda_{\max}\big(Q_{\Lcal}^*Q_{\Lcal}\big).
    \end{equation}
    By the orthonormality of $Q$, one obtains
    \begin{equation}\label{eqn::Q_LQ_L_definition_general_form}
    \begin{split}
        Q_{\Lcal}^*Q_{\Lcal} & = (Q^*SQ)^{-1}Q^*SSQ(Q^*SQ)^{-1} = \left((Q^*SQ)^2\right)^{-1},
    \end{split}
    \end{equation}
    which reduces to the following $\ell_2$-norm
    \begin{equation}\label{eqn::norm_of_dual_basis_standard_setting}
        \big\|Q_{\Lcal}\big\|_2 = \left(\sqrt{\lambda_{\min}\big(Q^*SQ\big)^2}\right)^{-1} = \frac{1}{|\lambda_{\min}(Q^*SQ)|}.
    \end{equation}
    Second, the eigenvalue condition number of $G$ is given by the inverse of the angle between the normalized left and right eigenvectors. Recall that $Q^*SHQ = LL^*$ is the Cholesky factorization, then, for a normalized eigenvector $\by$ of the reduced hermitian matrix \eqref{eqn::spectral_equivalent_hermitian_eigenvalue problem}, we have
    \begin{equation}
        \cond\big(\Tilde{\lambda},G\big) = \frac{\big\|\bw\big\|_2\big\|\bz\big\|_2}{\big|\bw^*\bz\big|} = \frac{\big\|\by^*L^{-1}\big\|_2\big\|L\by\big\|_2}{\big|\by^*L^{-1}L\by\big|} \leq \sqrt{\cond\big(Q^*SHQ\big)}.
    \end{equation}
    Since $SH$ is hermitian (see \eqref{introduction::quasi_hermitian_condition}) and $Q$ is unitary, the eigenvalues of $Q^*SHQ$ interlace those of $SH$. Therefore
    \begin{equation}
        \cond\big(\Tilde{\lambda},G\big) \leq \sqrt{\cond\big(SH\big)} = \sqrt{\cond\big(H\big)},
    \end{equation}
    where the last equality follows from Theorem \ref{theorem::condition_number_same}. Combining both estimates gives
    \begin{equation}\label{eqn::upper_bound_for_eigenvalue_conditioning_standard_setting}
        \big|\Tilde{\delta}\big|^{-1} \leq \dfrac{\sqrt{\cond\big(H\big)}}{\left|\lambda_{\min}\big(Q^*SQ\big)\right|}. 
    \end{equation}
\end{proof}
The denominator $|\lambda_{\min}(Q^*SQ)|$ does not depend on $\sigma_{\bv}$ or on the target Ritz value $\Tilde{\lambda}$. While this quantity may approach zero in degenerate cases, in practice it is typically well-conditioned (see below for a further discussion).

Since $\Ocal(\sigma_{\bu}) = \Ocal(\sigma_{\bv})$ and $|\Tilde{\delta}|^{-1}$ is bounded independently of $\sigma_{\bv}$ and $\Tilde{\lambda}$, the Ritz values converge quadratically. This is stipulated by the next lemma.
\begin{lemma}\label{theorem::hermitian_equivalent_rayleigh-ritz_quadratic}
    The Ritz values converge quadratically, in the sense that
    \begin{equation} \label{eqn::hermitian_equivalent_rayleigh-ritz_quadratic}
        \big| \lambda - \Tilde{\lambda} \big| \leq \kappa \cdot \Ocal\big(\sigma^2_{\bv}\big) \quad \text{with} \quad \kappa \defeq \dfrac{\sqrt{\cond\big(H\big)}\cdot\big\|H-\lambda I\big\|_2}{\left|\lambda_{\min}\big(Q^*SQ\big)\right|}.
    \end{equation}
\end{lemma}
The quadratic convergence of the Ritz-values is thus established. The bound depends primarily on the condition number of $H$ and the smallest eigenvalue in magnitude of $Q^*SQ$. Although this framework provides strong guarantees in terms of performance and convergence, difficulties arise when $\lambda_{\min}\big(Q^*SQ\big)$ is very close to zero. In that case the denominator may drastically amplify the bound. Such situations are rare in practice but cannot be ruled out \textit{a priori}. Because $S$ has the structure \eqref{introduction::pseudo_hermitian_condition} and $Q$ is unitary, Cauchy’s interlacing theorem ensures that the spectrum of $Q^*SQ$ is contained in the interval $[-1,1]$. Consequently, the denominator always acts to enlarge the upper bound $\kappa$. The next theorem sharpens this observation by providing an explicit interval in which the Ritz values must lie, showing how $\lambda_{\min}(Q^*SQ)$ effectively stretches it.
\begin{theorem}\label{theorem::field_of_values_ritz_hermitian_variant}
    Assume $Q^*SQ$ is non-singular. Then the Ritz values are all located within a closed interval, namely
    \begin{equation}\label{eqn::field_of_values_ritz_hermitian_variant}
        \Tilde{\lambda} \in \left[-\dfrac{\rho\big(SH\big)}{\big|\lambda_{\min}\big(Q^*SQ\big)\big|},\dfrac{\rho\big(SH\big)}{\big|\lambda_{\min}\big(Q^*SQ\big)\big|}\right].
    \end{equation}
\end{theorem}
\begin{proof}
 Since both matrices forming the Rayleigh-Quotient are hermitian, their fields of values coincide with their spectral intervals. 
    First, since $Q^*SQ$ is assumed non-singular, its spectrum is symmetric with respect to zero, and we obtain
    \begin{equation*}
        \quad \Fcal(Q^*SQ) \subset \left[-\left|\lambda_{\max}(Q^*SQ)\right|,-\left|\lambda_{\min}(Q^*SQ)\right|\right]\cup\left[\left|\lambda_{\min}(Q^*SQ)\right|,\left|\lambda_{\max}(Q^*SQ)\right|\right]
    \end{equation*}
    Second, the eigenvalues of $Q^*SHQ$ interlace with those of $SH$, since $Q$ is unitary. Therefore, we have $\Fcal(Q^*SHQ) \subset \big[\lambda_{\min}(SH),\rho(SH)\big].$
    By the result of \cite[chapter~1]{horn_topics_1991}, the Ritz values are contained in the field of values $\Fcal(Q^*SHQ) / \Fcal(Q^*SQ)$, which leads directly to \eqref{eqn::field_of_values_ritz_hermitian_variant}.
\end{proof}
This result complements Theorem \ref{theorem::hermitian_equivalent_rayleigh-ritz_quadratic}:
while the bound $\kappa$ already reveals how a small eigenvalue of $Q^*SQ$ can deteriorate convergence, Theorem \ref{theorem::field_of_values_ritz_hermitian_variant} shows that it also enlarges the admissible interval in which the Ritz values may lie. Since $A$ is a principal submatrix of $SH$ and that Lemma \ref{lemma::interval_of_real_eigenvalues} provides the exact location interval of the eigenvalues, we subsequently have $\lambda \in [-\rho(A),\rho(A)]\subseteq  [-\rho(SH),\rho(SH)] / |\lambda_{\min}\big(Q^*SQ\big)|$. Any Ritz-values falling outside this interval can be regarded as ``spurious values''. The domain where such spurious values appear expands as $\lambda_{\min}(Q^*SQ)$ decreases. 
\OLD{
\subsection{Non-hermitian Rayleigh-Quotient as back-up} This framework will only be used in the case of failure of the default hermitian eigenproblem where $Q^*SQ$ is singular. Here, we want to ensure that the alternative $M$ matrix is always non-singular and well-conditioned. One straightforward and practical choice that is very easy to check is setting $M \defeq \text{diag}(Q^*SQ)$. In that case, the dual basis becomes
\begin{equation} \label{eqn::second_dual_basis_pseudo-hermitian_identity}
    Q_{\Lcal} \defeq \left[SQ - Q \left(Q^*SQ - \text{diag}(Q^*SQ)\right)\right]\text{diag}^{-1}(Q^*SQ) ,
\end{equation}
and the Rayleigh-Quotient $G$ is given by
\begin{equation}\label{eqn::back_up_rayleigh_quotient}
    G = \text{diag}^{-1}(Q^*SQ)\left[Q^*SHQ - \left(Q^*SQ - \text{diag}(Q^*SQ)\right)Q^*HQ\right].
\end{equation}
As expected, this matrix $G$ is not hermitian, and has no hermitian spectral equivalence due to the appearance of $Q^*HQ$ in \eqref{eqn::back_up_rayleigh_quotient}. Hence, the Ritz-values are all complex. Therefore, we keep only the real parts in practice. In addition, the ideal equality \eqref{eqn::ideal_norm_equality_reached} does not apply here. Therefore, the approximation error of the eigenvalues is unlikely to decrease quadratically with respect to $\sigma_{\bv}$, which motivates resorting to this implementation only as a back-up option in case of failure of the default setting $M \defeq Q^*SQ$. In the eventual case where a zero entries appear in $\text{diag}(Q^*SQ)$, one can always replace it without breaking the projection properties of Rayleigh-Ritz. 

Figure \ref{fig::residual_evolution_real_medium} shows the evolution of the residual errors $\|H\Tilde{\bv}_i - \Tilde{\lambda}_i\hat{\bv}_i\|_2$ with respect to the number of iterations for two real world matrices generated by the Yambo code \cite{marini_yambo_2009}. Green curves show the convergence with the \textit{hermitian spectrally equivalent Rayleigh-Quotient} (\texttt{HEEVD}) and blue curves depict the convergence when using the alternative \textit{non-hermitian Rayleigh-Quotient} (\texttt{GEEV}). The red lines represent the evolution of the residuals in the standard version of ChASE when applied to a spectrally equivalent hermitian matrix generated artificially. 

For the first hundred smallest eigenpairs (Figure \ref{fig::first_nev_medium}), the rates of convergence for the three different implementations appear very similar. However, the \texttt{GEEV}-based implementation stagnates around a slightly higher residual error than the two other variants. Focusing on the last eigenpairs ranging from index $900$ to index $1000$, the second subfigure \ref{fig::last_nev_medium}  reveals a clear difference between the standard version  of ChASE and the pseudo-hermitian ones. While the standard implementation starts stagnating below $10^{-13}$ after the $4^{\text{th}}$ iteration, the pseudo-hermitian implementations require one to two more iterations to reach the same residual error. In practice, the parameter $\mu_{\nevex}$ is updated at each iteration to track only the remaining portion of non-converged eigenpairs. These observations hold for the larger test case illustrated in Figure~\ref{fig::residual_evolution_real_medium}.

In spite of the complexity of its convergence properties, the experiments show that the \textit{non-hermitian Rayleigh-Ritz} approach can be a viable method for expanding ChASE to solve for certain pseudo-hermitian Hamiltonians. This is confirmed in the further convergence experiments illustrated in Figure \ref{fig::convergence_experiment}. The \texttt{HEEVD} variant demonstrates a faster convergence while offering a cheaper numerical complexity and smaller memory requirement. That said, \texttt{GEEV} remains a suitable choice in certain cases. Moreover, we recall that the latter approach does not rely on the positive-definiteness assumption of $SH$ as stated in \eqref{introduction::quasi_hermitian_condition}. Hence, this implementation may be a good candidate in expanding ChASE functionalities to the more general case when $SH$ is indefinite.

\begin{figure}[h!]
\centering
\hfill
\begin{minipage}[l]{0.49\textwidth}
\includegraphics[scale=0.65]{figures/january/residuals_evolution_n=64512_double_first.pdf}
\subcaption{First $100^{\text{th}}$ }
\label{fig::first_nev_medium}
\end{minipage}
\hfill
\begin{minipage}[c]{0.49\textwidth}
\includegraphics[scale=0.65]{figures/january/residuals_evolution_n=64512_double_last.pdf}
\subcaption{Last $100^{\text{th}}$}
\label{fig::last_nev_medium}
\end{minipage}
\captionsetup{justification=centering}
\vspace{-0.5cm}
\caption{Evolution of the first $100^{\text{th}}$ vs. last $100^{\text{th}}$ residuals out of $\nev = 645$ smallest eigenpairs for a Molybdenum Disulfide pseudo-hermitian Hamiltonian of size $n = 64512$ generated with Yambo}
\label{fig::residual_evolution_real_medium}
\end{figure}}
\section{Numerical \wurw{experiments}}
\label{sec::benchmarks}
\OLD{In this section, we test the proposed extension of ChASE.} 
\NEW{\wurw{This section presents numerical experiments evaluating the proposed extension of ChASE on a distributed-GPU system, including convergence behavior and strong scaling benchmarks. The parallel implementation adopts the same strategy introduced in \cite{wu_advancing_2023}. All experiments are performed in single precision (FP32) on the JUPITER supercomputer at the J\"ulich Supercomputing Centre in Germany. Each node comprises four NVIDIA Grace–Hopper (GH200) superchips, each integrating a 72-core Grace CPU with a Hopper GPU (96 GB HBM3). GPUs are interconnected via NVLink~4, and nodes are connected through InfiniBand NDR. The theoretical peak performance of a single GPU for FP32 is 67 TFLOP/s (with tensor cores). The code is compiled with GCC~14.3.0, CUDA~13, and OpenMPI~5.0.8, with NCCL for communication.}}

The numerical experiments are conducted using \OLD{four}\NEW{\wurw{six}} pseudo-hermitian Hamiltonians \OLD{encoded}\NEW{\wurw{constructed}} by the Bethe-Salpeter equation and generated by the Yambo \cite{marini_yambo_2009} code. \NEW{Three of the Hamiltonians, \textsc{Si-1k}, \textsc{Si-11k} and \textsc{Si-39k}}, corresponds to the discretizations of a Silicon material (Si) with \NEW{three} different supercells, and have sizes \NEW{$2{,}944$}, $23{,}552$ and $79{,}488$, respectively. Similarly, the remaining \NEW{three} Hamiltonians, \NEW{\textsc{MoS$_2$-4k}, \textsc{MoS$_2$-32k} and \textsc{MoS$_2$-52k}}, correspond to a Molybdenum Disulfide (MoS$_2$) system and have sizes \NEW{$9{,}416$}, $64{,}512$ and $104{,}832$, respectively\NEW{\footnote{\wurw{We note that the naming convention (e.g., \textsc{52k}) refers to the dimension of the resonant block; the full pseudo-hermitian Hamiltonian has twice this dimension due to its block structure}.}}.

\NEW{\wurw{As a baseline, ChASE is used as a black-box solver without incorporating problem-specific physical information. We consider $\nev \in \{1\%,2\%,3\%\}\times n$, and for TDA, which involves only the resonant block, this corresponds to $\nev \in \{2\%,4\%,6\%\}\times m$ 
The size of the search space extension $\nex$ is chosen proportional to $\nev$ and depends on the target residual tolerance. Specifically, we set $\nex = 2/3 \times \nev $ for a tolerance of $10^{-4}$, and $\nex=\nev$ for a tighter tolerance of $10^{-5}$.}}
\OLD{As a start, ChASE is used as a black-box solver, without incorporating problem-specific physical information. Accordingly, the size of the extra space is always set to $\nex \defeq \nev$, and $\nev \in \{1\%,2\%,3\%\}\times n$. For all experiments, the MPI process grid is configured to be as close to square as possible.}

\subsection{Convergence Experiments}
\NEW{Figures \ref{fig::iteration_tol4} and \ref{fig::iteration_tol5} show the number of iterations required by ChASE to solve the six pseudo-hermitian eigenproblems at tolerance $10^{-4}$ and $10^{-5}$, respectively. The maximal number of iterations is set to $25$, and the Chebyshev polynomial degrees dynamically ranges from $8$ to $20$ depending on the actual convergence rate at each iteration (See Section 4, \cite{winkelmann_chase_2019}).  In both figures, the plain bars correspond to the full pseudo-hermitian Hamiltonian $H$, while the dashed bars correspond to the TDA resonant block $A$, solved with the standard ChASE hermitian solver \cite{wu_advancing_2023}.}
\begin{figure}[h!]
\centering
\hspace{-20pt}
\includegraphics[scale=1]{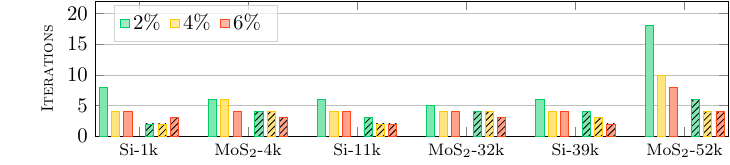}
\captionsetup{justification=centering}
\captionsetup{justification=centering}
\caption{Number of iterations with respect to the resonant size -- $\tol = 10^{-4}$ -- $\nex = 2/3 \times \nev $ -- Hamiltonian $H$ (plain bars) vs. Resonant $A$ (dashed bars)}
\label{fig::iteration_tol4}
\end{figure}
\begin{figure}[h!]
\centering
\hspace{-20pt}
\includegraphics[scale=1]{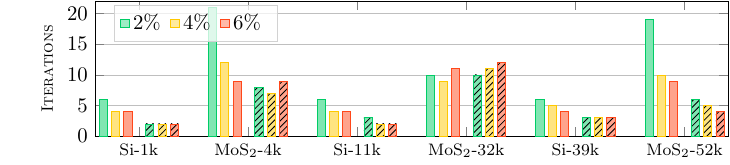}
\captionsetup{justification=centering}
\caption{Number of iterations with respect to the resonant size  -- $\tol = 10^{-5}$ -- $\nex = \nev $ -- Hamiltonian $H$ (plain bars) vs. Resonant $A$ (dashed bars)}
\label{fig::iteration_tol5}
\end{figure}

Overall, ChASE converges in fewer than $25$ iterations in all cases, fewer than 10 iterations for most of cases. The pseudo-hermitian variant requires slightly more iterations that the hermitian solver applied to the resonant TDA block, reflecting the effect of squaring the matrix in the filter step, which brings the convergence ratio of the target eigenvalues closer to one 
 (see \cite[Section 2]{napoli_estimating_2026}). In Fig.~\ref{fig::iteration_tol5}, the increased iteration count for \textsc{MoS$_2$} compared to \textsc{Si} is a direct consequence of higher spectral density near the smallest positive eigenvalue, where clustered eigenvalues reduce the filtering efficiency. 
 With the exception of \textsc{MoS$_2-32$k} for $\tol = 10^{-5}$, the number of iterations either stays constant or decreases with increasing $\nev$. This is consistent with the definition of ratio of convergence, where increasing $\nex$ may include spectral gaps that enhance the ratio of convergence for $|\rho|_{\lambda < \lambda_{\nev}}$ 
 , showing that enlarging the extra space is particularly important for tightly packed eigenvalues. 
 
 \wurw{In practice, the parameter $\nex$ can be tuned to improve convergence and overall performance. Automating the selection of $\nex$ based on spectral density
 is a promising direction for future work. In the present study, we adopted an universal choice of $\nex$ for all physical systems---and set $\mu_{\nevex}$ accordingly---to facilitate a consistent and transparent convergence comparison across all tests.}
 
 \subsection{Performance Experiments} \NEW{We evaluate the strong scaling of the extended ChASE solver on two pseudo-hermitian Hamiltonians in \NEW{FP32}\OLD{double precision} on \wurw{JUPITER, using up to 64 nodes (256 GPUs in total)}: a \textit{medium} test case with \textsc{Si-39k} matrix, and a \textit{large} test case with \textsc{MoS$_2$-52k} matrix. \wurw{The corresponding 2D MPI grid is configured as a square, which is the preferred layout for ChASE.} Each experiment is repeated five times, using the same parameters as for the convergence experiment of Figure \ref{fig::iteration_tol4} (i.e., $\tol = 10^{-4}$ and $\nex = 2/3\times\nev$).} In the figures, solid lines indicate the average over 5 repetitions, while shaded regions show the full range from the minimum to the maximum values. Black lines correspond to the ideal strong scaling. The color scheme for different $\nev$ values follows Figure \ref{fig::iteration_tol4} (i.e., $1\%$ in green, $2\%$ in orange, $3\%$ in red).
\begin{figure}[h!]
\centering
\hfill
\begin{minipage}[l]{0.48\textwidth}
\includegraphics[scale=0.325]{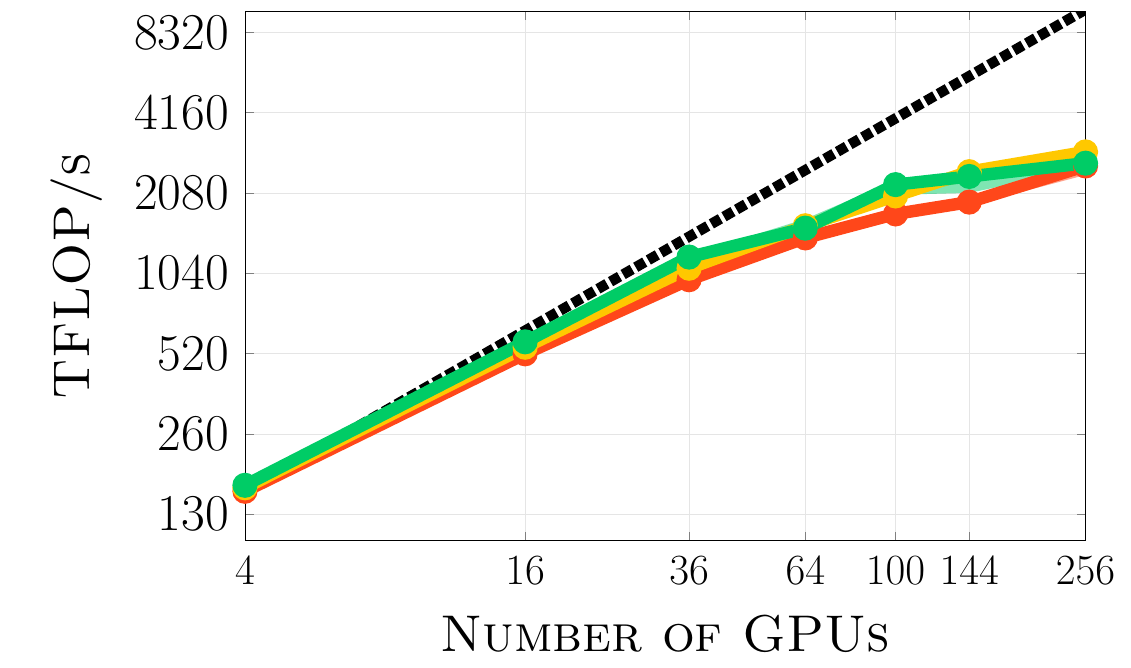}\captionsetup{justification=centering}
\end{minipage}
\hfill
\begin{minipage}[r]{0.48\textwidth}
\includegraphics[scale=0.325]{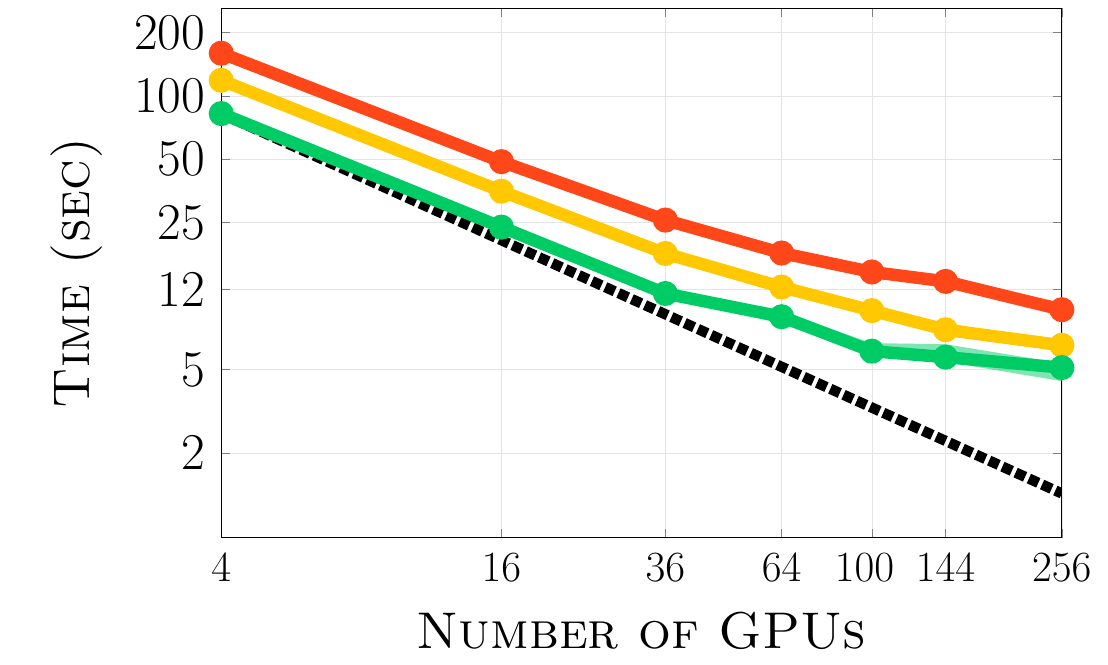}\captionsetup{justification=centering}
\end{minipage}
\hfill
\captionsetup{justification=centering}
\caption{Performance benchmark on the Si-39k matrix, $n = 79{,}488$ - $\nev = 794,1589,2384$.}
\label{fig::strong_scaling_medium}
\end{figure}

\begin{figure}[h!]
\centering
\hfill
\begin{minipage}[l]{0.48\textwidth}
\includegraphics[scale=0.325]{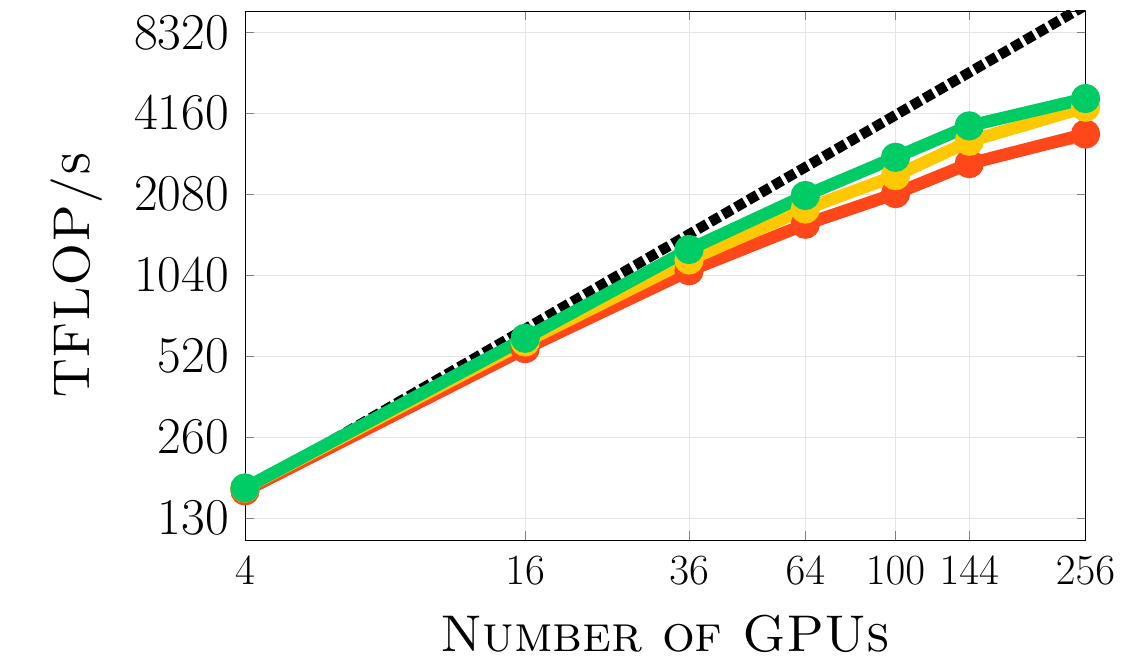}\captionsetup{justification=centering}
\end{minipage}
\hfill
\begin{minipage}[r]{0.48\textwidth}
\includegraphics[scale=0.325]{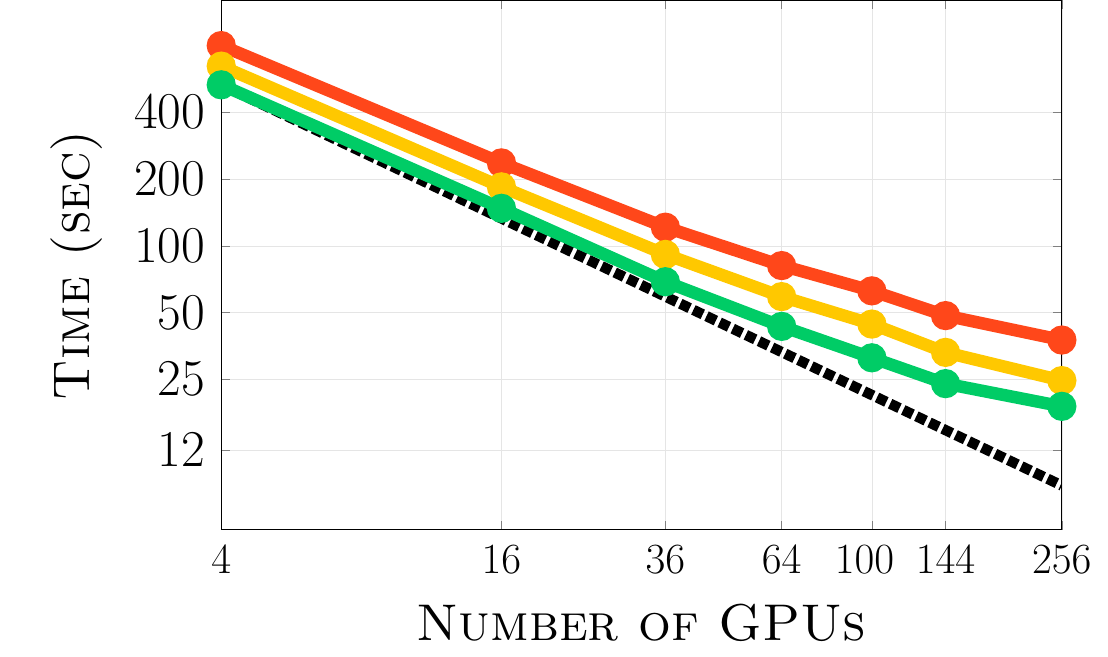}\captionsetup{justification=centering}
\end{minipage}
\hfill
\captionsetup{justification=centering}
\caption{Performance benchmark on the MoS$_2$-52k matrix, $n = 104{,}832$ - $\nev = 1048,2096,3144$}
\label{fig::strong_scaling_large}
\end{figure}
\NEW{\wurw{For the medium matrix ($n = 79{,}488$), Figure \ref{fig::strong_scaling_medium} shows that ChASE extracts $794$ of the smallest positive eigenpairs in $5.1$ seconds with $256$ GPUs, achieving $2.7$ PFLOP/s. When computing three times more eigenpairs ($\nev = 2{,}384$), the runtime doubles to $9.3$ seconds, with a comparable performance of $2.6$ PFLOP/s. For the large matrix ($n = 104{,}832$), Figure \ref{fig::strong_scaling_large} reports $3.5$ and $4.6$ PFLOP/s for $1{,}048$ and $3{,}144$ internal positive eigenpairs, respectively, with execution times ranging from $18.8$ to $37.3$ seconds on $256$ GPUs. For context, \cite[figure 4]{milev_performances_2025} reports a performance comparison experiment between SLEPc and ELPA on the Leonardo supercomputer equipped with A100 GPUs. ELPA requires over $7$ minutes to diagonalize a smaller matrix ($n = 72{,}576$) in single precision using $64$ GPUs, while SLEPc extracts $100$ eigenpairs in roughly $30$ seconds. On the JUPITER machine with H100 GPUs, which deliver usually twice the FP32 performance of A100s in practice, ChASE computes a larger number of eigenpairs in a comparable time scale. This highlights the excellent scalability and efficiency of ChASE on modern GPU-based architectures.}}

 \NEW{\wurw{The achieved effective FLOP/s relative to the theoretical peak performance on a single node (4 GPUs) exceeds 60\% in all cases, demonstrating efficient use of the hardware. In the strong-scaling experiments, as the number of GPUs increases to $256$, parallel efficiency for the larger case remains around 30\% (see Figure \ref{fig::par_efficiency}). Correspondingly, the average Streaming Multiprocessor (SM) utilization decreases from 88\% on a single node to 20\% on $256$ nodes, reflecting the reduced workload per GPU inherent to the strong-scaling regime. Despite this drop, ChASE continues to deliver consistent runtimes and robust performance, highlighting the scalability and effectiveness of the solver across large GPU clusters.}}

\begin{figure}[h!]
\centering
\hfill
\begin{minipage}[l]{0.48\textwidth}
\includegraphics[scale=0.325]{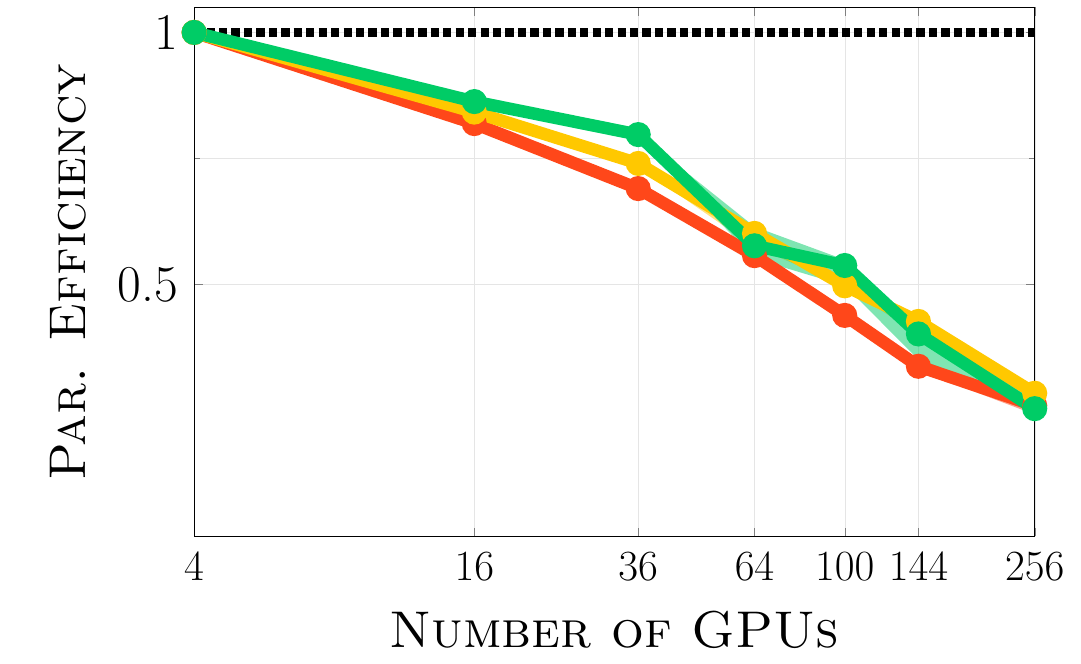}\captionsetup{justification=centering}
\end{minipage}
\hfill
\begin{minipage}[r]{0.48\textwidth}
\includegraphics[scale=0.325]{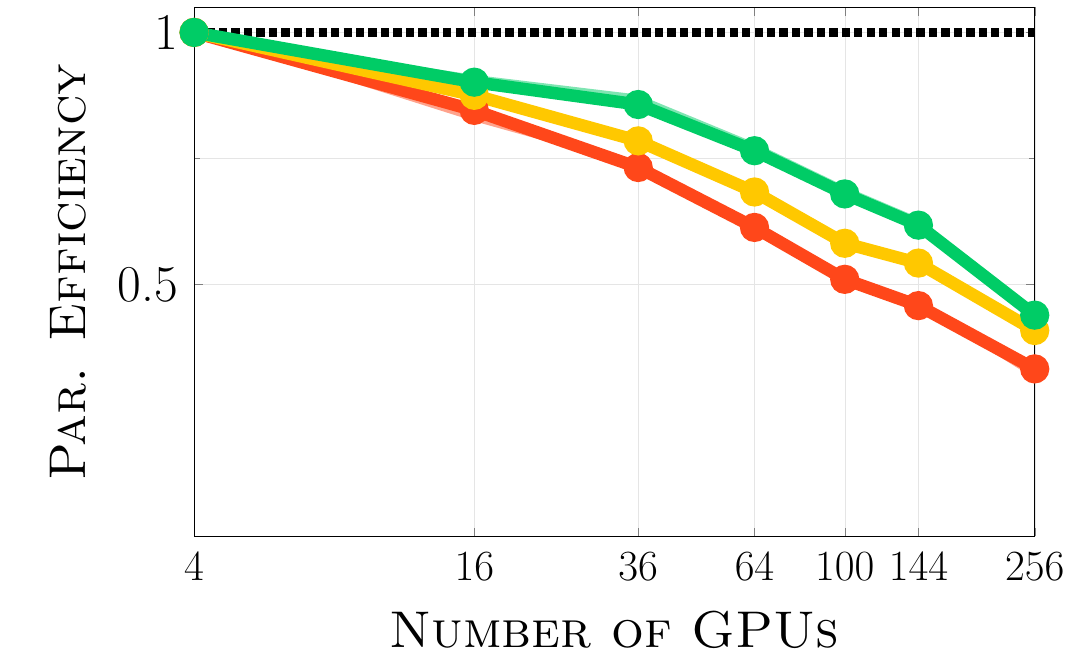}\captionsetup{justification=centering}
\end{minipage}
\hfill
\captionsetup{justification=centering}
\caption{Parallel Efficiency for the Si-39k (left) and MoS$_2$-52k (right) matrices.}
\label{fig::par_efficiency}
\end{figure}

For smaller matrices and when only a limited number of eigenpairs (e.g., at most around one hundred) are required, the results reported in \cite{milev_performances_2025} indicate that the Lanczos-based solver from SLEPc can achieve time-to-solution comparable to that of ChASE. \wurw{However, as the number of requested eigenpairs increases, Lanczos-based methods face significant performance limitations due to the cost of full reorthogonalization.} In contrast, direct solvers such as ELPA are highly efficient for full diagonalization, but may incur unnecessary computational cost when only a few thousand eigenpairs are of interest. In this landscape, ChASE—while itself an iterative subspace method—naturally targets an intermediate regime, offering a scalable and robust alternative for problems where neither Lanczos-based solvers nor full diagonalization are ideally matched to the computational requirements.

\section{Conclusion} \NEW{Computing the smallest eigenvalues} of large excitonic Hamiltonians is critical in many Condensed Matter applications and especially for the simulation of optoelectronic properties of complex materials, where the pseudo-hermitian formulations plays an important role in preventing the limiting approximation of the reduction to hermitian form (TDA). While recent numerical libraries have improved support for pseudo-hermitian solvers, the efficient calculation of several thousand \NEW{of the smallest positive} eigenpairs on modern heterogeneous architectures has remained an open challenge. This work extends ChASE to address this challenge by enabling the scalable computation of up to several thousand eigenpairs of large pseudo-hermitian Hamiltonians. \NEW{The essential filtering step is carried out using $H^2$ in order to fold the spectrum around zero and treat positive and negative eigenpairs symmetrically. By exploiting the structural properties of the pseudo-hermitian Hamiltonian, one half of the filtered subspace is directly recovered from the other, and }the parallel implementation of the Chebyshev filter minimizes global communication and achieves competitive performance and strong scalability through extensive use of optimized GPU kernels.  The proposed Rayleigh-Ritz formulation, based on an implicit dual basis, yields a Rayleigh-Quotient that is spectrally equivalent to an hermitian matrix, \NEW{preserves the characteristic positive-negative spectral symmetry of the Hamiltonian}, and \NEW{achieves} quadratic convergence of the Ritz values, similarly to the hermitian case. A rigorous numerical analysis is accompanied by a carefully designed set of experiments to demonstrate that ChASE provides an effective and scalable solution for large-scale pseudo-hermitian eigenproblems.

\bibliographystyle{siamplain}

\bibliography{no_url_references.bib}

\end{document}